\documentclass{amsart}
\usepackage{amscd,amsfonts,amsmath,amssymb,amsthm,curves,enumerate,graphicx,latexsym,multibox}
\usepackage{bbm}
\usepackage{mathrsfs}
\usepackage[mathscr]{eucal}
\usepackage{epsfig,epsf,xypic,epic}

\usepackage[T1]{fontenc}
\usepackage[sc]{mathpazo}

\newtheorem{thm}{Theorem}[section]
\newtheorem{lem}{Lemma}[section]

\newtheorem{prop}{Proposition}[section]
\newtheorem{defn}{Definition}[section]
\newtheorem{nb}{Remark}[section]

\numberwithin{equation}{section}

\begin{document}

\title[Mirror symmetry for toric Fano manifolds via SYZ]
{Mirror symmetry for toric Fano manifolds via SYZ transformations}
\author[K.-W. Chan]{Kwokwai Chan}
\address{The Institute of Mathematical Sciences and Department of Mathematics,
The Chinese University of Hong Kong, Shatin, Hong Kong}
\address{\textit{Current address: Department of Mathematics, Harvard
University, 1 Oxford Street, Cambridge, MA 02138, USA}}
\email{kwchan@math.harvard.edu}
\author[N.-C. Leung]{Naichung Conan Leung}
\address{The Institute of Mathematical Sciences and Department of Mathematics,
The Chinese University of Hong Kong, Shatin, Hong Kong}
\email{leung@ims.cuhk.edu.hk}

\begin{abstract}
We construct and apply Strominger-Yau-Zaslow mirror transformations
to understand the geometry of the mirror symmetry between toric Fano
manifolds and Landau-Ginzburg models.
\end{abstract}

\maketitle

\tableofcontents

\section{Introduction}

Mirror symmetry has been extended to the non-Calabi-Yau setting,
notably to Fano manifolds, by the works of
Givental~\cite{Givental94},~\cite{Givental96},~\cite{Givental97a},
Kontsevich~\cite{Kontsevich98} and Hori-Vafa~\cite{HV00}. If
$\bar{X}$ is a Fano manifold, then its mirror is conjectured to be a
pair $(Y,W)$, where $Y$ is a non-compact K\"{a}hler manifold and
$W:Y\rightarrow\mathbb{C}$ is a holomorphic Morse function. In the
physics literature, the pair $(Y,W)$ is called a
\emph{Landau-Ginzburg model}, and $W$ is called the
\emph{superpotential} of the model. One of the very first
mathematical predictions of this mirror symmetry is that there
should be an isomorphism between the small quantum cohomology ring
$QH^*(\bar{X})$ of $\bar{X}$ and the Jacobian ring $Jac(W)$ of the
function $W$. This has been verified (at least) for toric Fano
manifolds by the works of Batyrev~\cite{Batyrev93},
Givental~\cite{Givental97a} and many others. A version of the
\emph{Homological Mirror Symmetry Conjecture} has also been
formulated by Kontsevich~\cite{Kontsevich98}, which again has been
checked in many
cases~\cite{Seidel00},~\cite{Ueda04},~\cite{AKO04},~\cite{AKO05},
~\cite{Abouzaid05},~\cite{Abouzaid06}. However, no direct geometric
explanation for the mirror symmetry phenomenon for Fano manifolds
had been given, until the works of Cho-Oh \cite{CO03}, which showed
that, when $\bar{X}$ is a toric Fano manifold, the superpotential
$W$ can be computed in terms of the counting of Maslov index two
holomorphic discs in $\bar{X}$ with boundary in Lagrangian torus
fibers.

On the other hand, the celebrated \emph{Strominger-Yau-Zaslow (SYZ)
Conjecture}~\cite{SYZ96} suggested that mirror symmetry for
Calabi-Yau manifolds should be understood as a \emph{T-duality},
i.e. dualizing special Lagrangian torus fibrations, modified with
suitable quantum corrections. This will explain the geometry
underlying mirror symmetry \cite{Morrison96}. Recently, Gross and
Siebert~\cite{GS07} made a breakthrough in the study of this
conjecture, after earlier works of Fukaya~\cite{Fukaya02} and
Kontsevich-Soibelman~\cite{KS04}. It is expected that their program
will finally provide a very explicit and geometric way to see how
mirror symmetry works for \textit{both} Calabi-Yau and
non-Calabi-Yau manifolds (more precisely, for varieties with
effective anticanonical class). On the other hand, in
\cite{Auroux07}, Auroux started his program which is aimed at
understanding mirror symmetry in the non-Calabi-Yau setting by
applying the SYZ approach. More precisely, he studied the mirror
symmetry between a general compact K\"{a}hler manifold equipped with
an anticanonical divisor and a Landau-Ginzburg model, and
investigated how the superpotential can be computed in terms
holomorphic discs counting on the compact K\"{a}hler manifold. In
particular, this includes the mirror symmetry for toric Fano
manifolds as a special case.\\

In this paper, we shall again follow the SYZ philosophy and study
the mirror symmetry phenomenon for toric Fano manifolds by using
T-duality. The main point of this work, which is also the crucial
difference between this and previous works, is that, explicit
transformations, which we call \textit{SYZ mirror transformations},
are constructed and used to understand the results (e.g.
$QH^*(\bar{X})\cong Jac(W)$) implied by mirror symmetry. From this
perspective, this paper may be regarded as a sequel to the second
author's work \cite{Leung00}, where semi-flat SYZ mirror
transformations (i.e. fiberwise real Fourier-Mukai transforms) were
used to study mirror symmetry for semi-flat Calabi-Yau manifolds.
While in that case, quantum corrections do not arise because the
Lagrangian torus fibrations are smooth (i.e. they are fiber
bundles), we will have to deal with quantum corrections in the toric
Fano case.

However, we shall emphasize that the quantum corrections which arise
in the toric Fano case are only due to contributions from the
anticanonical toric divisor (the toric boundary); correspondingly,
the Lagrangian torus fibrations do \textit{not} have proper singular
fibers (i.e. singular fibers which are contained in the complement
of the anticanonical divisor), so that their bases are affine
manifolds with boundaries but \textit{without singularities}. This
is simpler than the general non-Calabi-Yau case treated by
Gross-Siebert~\cite{GS07} and Auroux~\cite{Auroux07}, where further
quantum corrections could arise, due to the fact that general
Lagrangian torus fibrations do admit proper singular fibers, so that
their bases are affine manifolds with \textit{both} boundaries and
singularities. Hence, the toric Fano case is in-between the
semi-flat case, which corresponds to nonsingular affine manifolds
without boundary, and the general case. In particular, in the toric
Fano case, we do not need to worry about wall-crossing phenomena,
and this is one of the reasons why we can construct the SYZ mirror
transformations explicitly as fiberwise Fourier-type transforms,
much like what was done in the semi-flat case \cite{Leung00}.
(Another major reason is that holomorphic discs in toric manifolds
with boundary in Lagrangian torus fibers are completely classified
by Cho-Oh \cite{CO03}.) It is interesting to generalize the results
here to non-toric settings, but, certainly, much work needs to be
done before we can see how SYZ mirror transformations are
constructed and used in the general case. For more detailed
discussions of mirror symmetry and the wall-crossing phenomena in
non-toric situations , we refer the reader to the works of
Gross-Siebert~\cite{GS07} and Auroux~\cite{Auroux07}.\\

What follows is an outline of our main results. We will focus on one
half of the mirror symmetry between a complex $n$-dimensional toric
Fano manifold $\bar{X}$ and the mirror Landau-Ginzburg model
$(Y,W)$, namely, the correspondence between the symplectic geometry
(A-model) of $\bar{X}$ and the complex geometry (B-model) of
$(Y,W)$.

To describe our results, let us fix some notations first. Let
$N\cong\mathbb{Z}^n$ be a lattice and
$M=N^\vee=\textrm{Hom}(N,\mathbb{Z})$ the dual lattice. Also let
$N_\mathbb{R}=N\otimes_\mathbb{Z}\mathbb{R}$,
$M_\mathbb{R}=M\otimes_\mathbb{Z}\mathbb{R}$, and denote by
$\langle\cdot,\cdot\rangle:M_\mathbb{R}\times
N_\mathbb{R}\rightarrow\mathbb{R}$ the dual pairing. Let $\bar{X}$
be a toric Fano manifold, i.e. a smooth projective toric variety
such that the anticanonical line bundle $K_{\bar{X}}$ is ample. Let
$v_1,\ldots,v_d\in N$ be the primitive generators of the
1-dimensional cones of the fan $\Sigma$ defining $\bar{X}$. Then a
polytope $\bar{P}\subset M_\mathbb{R}$ defined by the inequalities
$$\langle x,v_i\rangle\geq\lambda_i,\quad i=1,\ldots,d,$$
and with normal fan $\Sigma$, associates a K\"{a}hler structure
$\omega_{\bar{X}}$ to $\bar{X}$. Physicists \cite{HV00} predicted
that the mirror of ($\bar{X}$, $\omega_{\bar{X}}$) is given by the
pair $(Y,W)$, where $Y$, which we call \emph{Hori-Vafa's mirror
manifold}, is biholomorphic to the non-compact K\"{a}hler manifold
$(\mathbb{C}^*)^n$, and $W:Y\rightarrow\mathbb{C}$ is the Laurent
polynomial
$$e^{\lambda_1}z^{v_1}+\ldots+e^{\lambda_d}z^{v_d},$$
where $z_1,\ldots,z_n$ are the standard complex coordinates of
$Y\cong(\mathbb{C}^*)^n$ and $z^v$ denotes the monomial
$z_1^{v^1}\ldots z_n^{v^n}$ if $v=(v^1,\ldots,v^n)\in
N\cong\mathbb{Z}^n$.

The symplectic manifold $(\bar{X},\omega_{\bar{X}})$ admits a
Hamiltonian action by the torus $T_N=N_\mathbb{R}/N$, and the
corresponding moment map $\mu_{\bar{X}}:\bar{X}\rightarrow\bar{P}$
is naturally a Lagrangian torus fibration. While this fibration is
singular (with collapsed fibers) along $\partial\bar{P}$, the
restriction to the open dense $T_N$-orbit $X\subset\bar{X}$ is a
Lagrangian torus bundle
$$\mu=\mu_{\bar{X}}|_X:X\rightarrow P,$$ where
$P=\bar{P}\setminus\partial\bar{P}$ is the interior of the polytope
$\bar{P}$.\footnote{$\mu:X\rightarrow P$ is a special Lagrangian
fibration if we equip $X$ with the standard holomorphic volume form
on $(\mathbb{C}^*)^n$, so that $X$ becomes an almost Calabi-Yau
manifold. See Definition 2.1 and Lemma 4.1 in Auroux
\cite{Auroux07}.} Applying T-duality and the \emph{semi-flat SYZ
mirror transformation} (see Definition~\ref{def3.1}) to this torus
bundle, we can, as suggested by the SYZ philosophy, obtain the
mirror manifold $Y$ (see Proposition~\ref{prop3.1} and
Proposition~\ref{prop3.2}).\footnote{In fact, we prove that
T-duality gives a bounded domain in the Hori-Vafa mirror manifold.
This result also appeared in Auroux's paper (\cite{Auroux07},
Proposition 4.2).} However, we are not going to get the
superpotential $W:Y\rightarrow\mathbb{C}$ because we have ignored
the anticanonical toric divisor $D_\infty=\bigcup_{i=1}^d
D_i=\bar{X}\setminus X$, and hence quantum corrections. Here, for
$i=1,\ldots,d$, $D_i$ denotes the toric prime divisor which
corresponds to $v_i\in N$. To recapture the quantum corrections, we
consider the (trivial) $\mathbb{Z}^n$-cover
$$\pi:LX=X\times N\rightarrow X$$
and various functions on it.\footnote{In the expository paper
\cite{CL08}, we interpreted LX as a finite dimensional subspace of
the free loop space $\mathcal{L}\bar{X}$ of $\bar{X}$.} Let
$\mathcal{K}(\bar{X})\subset H^2(\bar{X},\mathbb{R})$ be the
K\"{a}hler cone of $\bar{X}$. For each
$q=(q_1,\ldots,q_l)\in\mathcal{K}(\bar{X})$ (here
$l=d-n=\textrm{Picard number of $\bar{X}$}$), we define a
$T_N$-invariant function $\Phi_q:LX\rightarrow\mathbb{R}$ in terms
of the counting of Maslov index two holomorphic discs in $\bar{X}$
with boundary in Lagrangian torus fibers of $\mu:X\rightarrow P$
(see Definition~\ref{def2.1} and Remark~\ref{rmk2.1}). If we further
assume that $\bar{X}$ is a product of projective spaces, then this
family of functions $\{\Phi_q\}_{q\in\mathcal{K}(\bar{X})}\subset
C^\infty(LX)$ can be used to compute the small quantum cohomology
ring $QH^*(\bar{X})$ of $\bar{X}$ as follows (see Section~\ref{sec2}
for details).
\begin{prop}\label{prop1.1}$\mbox{}$
\begin{enumerate}
\item[1.] The logarithmic derivatives of $\Phi_q$, with respect to
$q_a$, for $a=1,\ldots,l$, are given by
$$q_a\frac{\partial\Phi_q}{\partial q_a}=\Phi_q\star\Psi_{n+a}.$$
Here, for each $i=1,\ldots,d$, the function
$\Psi_i:LX\rightarrow\mathbb{R}$ is defined, in terms of the
counting of Maslov index two holomorphic discs in $\bar{X}$ with
boundary in Lagrangian torus fibers which intersect the toric prime
divisor $D_i$ at an interior point (see the statement of
Proposition~\ref{prop2.1} and the subsequent discussion), and
$\star$ denotes a convolution product of functions on $LX$ with
respect to the lattice $N$.
\item[2.] We have a natural isomorphism of $\mathbb{C}$-algebras
\begin{equation}\label{isom1.1}
QH^*(\bar{X})\cong\mathbb{C}[\Psi_1^{\pm1},\ldots,\Psi_n^{\pm1}]/\mathcal{L},
\end{equation}
where $\mathbb{C}[\Psi_1^{\pm1},\ldots,\Psi_n^{\pm1}]$ is the
polynomial algebra generated by $\Psi_1^{\pm1},\ldots,\Psi_n^{\pm}$
with respect to the convolution product $\star$, and $\mathcal{L}$
is the ideal generated by linear relations that are defined by the
linear equivalence among the toric divisors $D_1,\ldots,D_d$,
provided that $\bar{X}$ is a product of projective spaces.
\end{enumerate}
\end{prop}
The proof of the above isomorphism (\ref{isom1.1}) given in
Subsection~\ref{subsec2.1} will be combinatorial in nature and is
done by a simple computation of certain Gromov-Witten invariants.
While this result may follow easily from known results in the
literature, we choose to include an elementary proof to make this
paper more self-contained. Our proof relies on the assumption that
$\bar{X}$ is a product of projective spaces. However, the more
important reason for us to impose such a strong assumption is that,
when $\bar{X}$ is a product of projective spaces, there is a better
way to understand the geometry underlying the isomorphism
(\ref{isom1.1}) by using \textit{tropical geometry}. A brief
explanation is now in order. More details can be found in
Subsection~\ref{subsec2.2}.

Suppose that $\bar{X}$ is a product of projective spaces. We first
define a tropical analog of the small quantum cohomology ring of
$\bar{X}$, call it $QH^*_{trop}(\bar{X})$. The results of
Mikhalkin~\cite{Mikhalkin03} and Nishinou-Siebert~\cite{NS04}
provided a one-to-one correspondence between those holomorphic
curves in $\bar{X}$ which have contribution to the quantum product
in $QH^*(\bar{X})$ and those tropical curves in $N_\mathbb{R}$ which
have contribution to the tropical quantum product in
$QH^*_{trop}(\bar{X})$. From this follows the natural isomorphism
$$QH^*(\bar{X})\cong QH^*_{trop}(\bar{X}).$$
Next comes a simple but crucial observation: \textit{each tropical
curve which contributes to the tropical quantum product in
$QH^*_{trop}(\bar{X})$ can be obtained by gluing tropical
discs}.\footnote{More recently, Gross \cite{Gross09} generalized
this idea further to give a tropical interpretation of the
\textit{big} quantum cohomology of $\mathbb{P}^2$.} Now, making use
of the fundamental results of Cho and Oh~\cite{CO03} on the
classification of holomorphic discs in toric Fano manifolds, we get
a one-to-one correspondence between the relevant tropical discs and
the Maslov index two holomorphic discs in $\bar{X}$ with boundary in
Lagrangian torus fibers of $\mu:X\rightarrow P$. The latter were
used to define the functions $\Psi_i$'s. So we naturally have
another canonical isomorphism
$$QH^*_{trop}(\bar{X})\cong\mathbb{C}[\Psi_1^{\pm1},\ldots,\Psi_n^{\pm1}]/\mathcal{L}.$$

Hence, by factoring through the tropical quantum cohomology ring
$QH^*_{trop}(\bar{X})$ and using the correspondence between
symplectic geometry (holomorphic curves and discs) of $\bar{X}$ and
tropical geometry (tropical curves and discs) of $N_\mathbb{R}$, we
obtain a more conceptual and geometric understanding of the
isomorphism (\ref{isom1.1}). This is in line with the general
philosophy advocated in the Gross-Siebert program \cite{GS07}.

Notice that all these can only be done for products of projective
spaces because, as is well known, tropical geometry cannot be used
to count curves which have irreducible components mapping to the
toric boundary divisor, and if $\bar{X}$ is not a product of
projective spaces, those curves \textit{do} contribute to
$QH^*(\bar{X})$ (see Example 3 in Section~\ref{sec4}). This is the
main reason why we confine ourselves to the case of products of
projective spaces, although the isomorphism (\ref{isom1.1}) holds
for all toric Fano manifolds (see Remark~\ref{rmk2.3}).\\

Now we come to the upshot of this paper, namely, we can explicitly
construct and apply SYZ mirror transformations to understand the
mirror symmetry between $\bar{X}$ and $(Y,W)$. We shall define the
SYZ mirror transformation $\mathcal{F}$ for the toric Fano manifold
$\bar{X}$ as \emph{a combination of the semi-flat SYZ mirror
transformation and taking fiberwise Fourier series} (see
Definition~\ref{def3.2} for the precise definition). Our first
result says that the SYZ mirror transformation of $\Phi_q$ is
precisely the exponential of the superpotential $W$, i.e.
$\mathcal{F}(\Phi_q)=\exp(W)$. Then, by proving that the SYZ mirror
transformation $\mathcal{F}(\Psi_i)$ of the function $\Psi_i$ is
nothing but the monomial $e^{\lambda_i}z^{v_i}$, for $i=1,\ldots,d$,
we show that $\mathcal{F}$ exhibits a natural and canonical
isomorphism between the small quantum cohomology ring
$QH^*(\bar{X})$ and the Jacobian ring $Jac(W)$, which takes the
quantum product $\ast$ (which can now, by Proposition~\ref{prop1.1},
be realized as the convolution product $\star$) to the ordinary
product of Laurent polynomials, just as what classical Fourier
series do. This is our main result (see Section~\ref{sec3}):
\begin{thm}\label{main_thm}$\mbox{}$
\begin{enumerate}
\item[1.] The SYZ mirror transformation of the function
$\Phi_q\in C^\infty(LX)$, defined in terms of the counting of Maslov
index two holomorphic discs in $\bar{X}$ with boundary in Lagrangian
torus fibers, is the exponential of the superpotential $W$ on the
mirror manifold $Y$, i.e.
$$\mathcal{F}(\Phi_q)=e^W.$$
Furthermore, we can incorporate the symplectic structure
$\omega_X=\omega_{\bar{X}}|_X$ on $X$ to give the holomorphic volume
form on the Landau-Ginzburg model $(Y,W)$ through the SYZ mirror
transformation $\mathcal{F}$, in the sense that,
\begin{equation*}
\mathcal{F}(\Phi_q e^{\sqrt{-1}\omega_X})=e^W\Omega_Y.
\end{equation*}
\item[2.] The SYZ mirror transformation gives a canonical isomorphism
of $\mathbb{C}$-algebras
\begin{equation*}\label{isom1.2}
\mathcal{F}:QH^*(\bar{X})\overset{\cong}{\longrightarrow} Jac(W),
\end{equation*}
provided that $\bar{X}$ is a product of projective spaces.
\end{enumerate}
\end{thm}
\noindent Here we view $\Phi_q e^{\sqrt{-1}\omega_X}$ as the
\emph{symplectic structure corrected by Maslov index two holomorphic
discs},\footnote{In \cite{CL08}, we rewrote the function $\Phi_q$ as
$\exp(\Psi_1+\ldots+\Psi_d)$, so that $\Phi_q
e^{\sqrt{-1}\omega_X}=\exp(\sqrt{-1}\omega_X+\Psi_1+\ldots+\Psi_d)$
and the formula in part 1. of Theorem~\ref{main_thm} becomes
$\mathcal{F}(e^{\sqrt{-1}\omega_X+\Psi_1+\ldots+\Psi_d})=e^W\Omega_Y$.
May be it is more appropriate to call
$\sqrt{-1}\omega_X+\Psi_1+\ldots+\Psi_d\in\Omega^2(LX)\oplus\Omega^0(LX)$
the symplectic structure corrected by Maslox index two holomorphic
discs.} and $e^W\Omega_Y$ as the holomorphic volume form of the
Landau-Ginzburg model $(Y,W)$.

As mentioned at the beginning, the existence of an isomorphism
$QH^*(\bar{X})\cong Jac(W)$ is not a new result, and was established
before by the works of Batyrev~\cite{Batyrev93} and
Givental~\cite{Givental97a}. However, we shall emphasize that the
key point here is that there is an isomorphism which is realized by
an explicit Fourier-type transformation, namely, the SYZ mirror
transformation $\mathcal{F}$. This hopefully provides a more
conceptual understanding of what is going on.

In \cite{FOOO08a} (Section 5), Fukaya-Oh-Ohta-Ono studied the
isomorphism $QH^*(\bar{X})\cong Jac(W)$ from the point of view of
Lagrangian Floer theory. They worked over the Novikov ring, instead
of $\mathbb{C}$, and gave a proof (Theorem 1.9) of this isomorphism
(over the Novikov ring) for all toric Fano manifolds basing on
Batyrev's formulas for presentations of the small quantum cohomology
rings of toric manifolds and Givental's mirror theorem
\cite{Givental97a}. Their proof was also combinatorial in nature,
but they claimed that a geometric proof would appear in a sequel of
\cite{FOOO08a}. A brief history and a more detailed discussion of
the proof of the isomorphism were also contained in Remark 1.10 of
\cite{FOOO08a}. See also the sequel \cite{FOOO08b}.\\

The rest of this paper is organized as follows. In the next section,
we define the family of functions
$\{\Phi_q\}_{q\in\mathcal{K}(\bar{X})}$ in terms of the counting of
Maslov index two holomorphic discs and give a combinatorial proof
Proposition~\ref{prop1.1}, which is followed by a discussion of the
role played by tropical geometry. The heart of this paper is
Section~\ref{sec3}, where we construct explicitly the SYZ mirror
transformation $\mathcal{F}$ for a toric Fano manifold $\bar{X}$ and
show that it indeed transforms the symplectic structure of $\bar{X}$
to the complex structure of $(Y,W)$, and vice versa. This is the
first part of Theorem~\ref{main_thm}. We then move on to prove the
second part, which shows how the SYZ mirror transformation
$\mathcal{F}$ can realize the isomorphism $QH^*(\bar{X})\cong
Jac(W)$. Section~\ref{sec4} contains some examples. We conclude with
some discussions in the final section.

\section{Maslov index two holomorphic discs and $QH^*(\bar{X})$}\label{sec2}

In the first part of this section, we define the functions $\Phi_q$,
$q\in\mathcal{K}(\bar{X})$, and $\Psi_1,\ldots,\Psi_d$ on $LX$ in
terms of the counting of Maslov index two holomorphic discs in
$\bar{X}$ with boundary in Lagrangian torus fibers of the moment map
$\mu:X\rightarrow P$, and show how they can be used to compute the
small quantum cohomology ring $QH^*(\bar{X})$ in the case when
$\bar{X}$ is a product of projective spaces. In particular, we
demonstrate how the quantum product can be realized as a convolution
product (part 2. of Proposition~\ref{prop1.1}). In the second part,
we explain the geometry of these results by using tropical geometry.

\subsection{Computing $QH^*(\bar{X})$ in terms of functions on
$LX$}\label{subsec2.1}

Recall that the primitive generators of the 1-dimensional cones of
the fan $\Sigma$ defining the toric Fano manifold $\bar{X}$ are
denoted by $v_1,\ldots,v_d\in N$. Without loss of generality, we can
assume that $v_1=e_1,\ldots,v_n=e_n$ is the standard basis of
$N\cong\mathbb{Z}^n$. The map
$$\partial:\mathbb{Z}^d\rightarrow N,\ (k_1,\ldots,k_d)\mapsto\sum_{i=1}^d k_iv_i$$
is surjective since $\bar{X}$ is compact. Let $K$ be the kernel of
$\partial$, so that the sequence
\begin{equation}\label{seq2.1}
0\longrightarrow K\overset{\iota}{\longrightarrow}\mathbb{Z}^d
\overset{\partial}{\longrightarrow}N\longrightarrow0
\end{equation}
is exact (see, for example, Appendix 1 in the book of Guillemin
\cite{Guillemin94}).

Now consider the K\"{a}hler cone $\mathcal{K}(\bar{X})\subset
H^2(\bar{X},\mathbb{R})$ of $\bar{X}$, and let
$q_1,\ldots,q_l\in\mathbb{R}_{>0}$ ($l=d-n$) be the coordinates of
$\mathcal{K}(\bar{X})$. For each
$q=(q_1,\ldots,q_l)\in\mathcal{K}(\bar{X})$, we choose $\bar{P}$ to
be the polytope defined by
$$\bar{P}=\{x\in M_\mathbb{R}:\langle x,v_i\rangle\geq\lambda_i,\quad i=1,\ldots,d\}$$
with $\lambda_i=0$, for $i=1,\ldots,n$, and $\lambda_{n+a}=\log
q_a$, for $a=1,\ldots,l$. This associates a K\"{a}hler structure
$\omega_{\bar{X}}$ to $\bar{X}$.
\begin{nb}\label{rmk2.0}
Let $\bar{P}$ be the polytope defined by the inequalities
$$\langle x,v_i\rangle\geq\lambda_i,\quad i=1,\ldots,d.$$
Also let
$$Q_1=(Q_{11},\ldots,Q_{d1}),\ldots,Q_l=(Q_{1l},\ldots,Q_{dl})\in\mathbb{Z}^d$$
be a $\mathbb{Z}$-basis of $K$. Then the coordinates
$q=(q_1,\ldots,q_l)\in\mathcal{K}(\bar{X})$ of the K\"{a}hler cone
are given by $q_a=e^{-r_a}$, where
$$r_a=-\sum_{i=1}^d Q_{ia}\lambda_i,$$
for $a=1,\ldots,l$. Hence, different choices of the
$\mathbb{Z}$-basis of $K$ and the constants
$\lambda_1,\ldots,\lambda_d$ can give rise to the same K\"{a}hler
structure parametrized by $q\in\mathcal{K}(\bar{X})$. We choose the
$\mathbb{Z}$-basis $\{Q_1,\ldots,Q_l\}$ of $K$ such that
$(Q_{n+a,b})_{1\leq a,b\leq l}=\textrm{Id}_{l\times l}$, and the
constants $\lambda_1,\ldots,\lambda_d$ such that
$\lambda_1=\ldots=\lambda_n=0$.
\end{nb}
Recall that $\mu:X\rightarrow P$ is the restriction of the moment
map $\mu_{\bar{X}}:\bar{X}\rightarrow\bar{P}$ to the open dense
$T_N$-orbit $X\subset\bar{X}$, where $P$ is the interior of the
polytope $\bar{P}$. For a point $x\in P$, we let
$L_x=\mu^{-1}(x)\subset X$ be the Lagrangian torus fiber over $x$.
Then the groups $H_2(\bar{X},\mathbb{Z})$, $\pi_2(\bar{X},L_x)$ and
$\pi_1(L_x)$ can be identified canonically with $K$, $\mathbb{Z}^d$
and $N$ respectively, so that the exact sequence (\ref{seq2.1})
above coincides with the following exact sequence of homotopy groups
associated to the pair $(\bar{X},L_x)$:
$$0\longrightarrow H_2(\bar{X},\mathbb{Z})\overset{\iota}{\longrightarrow}
\pi_2(\bar{X},L_x)\overset{\partial}{\longrightarrow}\pi_1(L_x)\longrightarrow0.$$
To proceed, we shall recall some of the fundamental results of
Cho-Oh~\cite{CO03} on the classification of holomorphic discs in
$(\bar{X},L_x)$:
\begin{thm}[Theorem 5.2 and Theorem
8.1 in Cho-Oh~\cite{CO03}]\label{cho-oh} $\pi_2(\bar{X},L_x)$ is
generated by $d$ Maslov index two classes
$\beta_1,\ldots,\beta_d\in\pi_2(\bar{X},L_x)$, which are represented
by holomorphic discs with boundary in $L_x$. Moreover, given a point
$p\in L_x$, then, for each $i=1,\ldots,d$, there is a unique (up to
automorphism of the domain) Maslov index two holomorphic disc
$\varphi_i:(D^2,\partial D^2)\rightarrow(\bar{X},L_x)$ in the class
$\beta_i$ whose boundary passes through $p$,\footnote{Another way to
state this result: Let $\mathcal{M}_1(L_x,\beta_i)$ be the moduli
space of holomorphic discs in $(\bar{X},L_x)$ in the class $\beta_i$
and with one boundary marked point. In the toric Fano case,
$\mathcal{M}_1(L_x,\beta_i)$ is a smooth compact manifold of real
dimension $n$. Let $ev:\mathcal{M}_1(L_x,\beta_i)\rightarrow L_x$ be
the evaluation map at the boundary marked point. Then we have
$ev_*[\mathcal{M}_1(L_x,\beta_i)]=[L_x]$ as $n$-cycles in $L$. See
Cho-Oh \cite{CO03} and Auroux \cite{Auroux07} for details.} and the
symplectic area of $\varphi_i$ is given by
\begin{equation}\label{area}
\textrm{Area}(\varphi_i)=\int_{\beta_i}\omega_{\bar{X}}=\int_{D^2}\varphi_i^*\omega_{\bar{X}}=2\pi(\langle
x,v_i\rangle-\lambda_i).
\end{equation}
\end{thm}
Furthermore, for each $i=1,\ldots,d$, the disc $\varphi_i$
intersects the toric prime divisor $D_i$ at a unique interior point.
(We can in fact choose the parametrization of $\varphi_i$ so that
$\varphi_i(0)\in D_i$.) Indeed, a result of Cho-Oh (Theorem 5.1 in
\cite{CO03}) says that, the Maslov index of a holomorphic disc
$\varphi:(D^2,\partial D^2)\rightarrow(\bar{X},L_x)$ representing a
class $\beta\in\pi_2(\bar{X},L_x)$ is given by twice the algebraic
intersection number $\beta\cdot D_\infty$, where
$D_\infty=\bigcup_{i=1}^dD_i$ is the toric boundary divisor (see
also Auroux~\cite{Auroux07}, Lemma 3.1).

Let $LX$ be the product $X\times N$. We view $LX$ as a (trivial)
$\mathbb{Z}^n$-cover over $X$:
$$\pi:LX=X\times N\rightarrow X,$$
and we equip $LX$ with the symplectic structure $\pi^*(\omega_X)$,
so that it becomes a symplectic manifold. We are now in a position
to define $\Phi_q$.
\begin{defn}\label{def2.1}
Let $q=(q_1,\ldots,q_l)\in\mathcal{K}(\bar{X})$. The function
$\Phi_q:LX\rightarrow\mathbb{R}$ is defined as follows. For
$(p,v)\in LX=X\times N$, let $x=\mu(p)\in P$ and $L_x=\mu^{-1}(x)$
be the Lagrangian torus fiber containing $p$. Denote by
$$\pi_2^+(\bar{X},L_x)=\Big\{\sum_{i=1}^d
k_i\beta_i\in\pi_2(\bar{X},L_x):k_i\in\mathbb{Z}_{\geq0},\
i=1,\ldots,d\Big\}$$ the positive cone generated by the Maslov index
two classes $\beta_1,\ldots,\beta_d$ which are represented by
holomorphic discs with boundary in $L_x$. For $\beta=\sum_{i=1}^d
k_i\beta_i\in\pi_2^+(\bar{X},L_x)$, we denote by $w(\beta)$ the
number $k_1!\ldots k_d!$. Then set
$$\Phi_q(p,v)=\sum_{\beta\in\pi_2^+(\bar{X},L_x),\ \partial\beta=v}
\frac{1}{w(\beta)}e^{-\frac{1}{2\pi}\int_\beta\omega_{\bar{X}}}.$$
\end{defn}
\begin{nb}\label{rmk2.1}$\mbox{}$
\begin{enumerate}
\item[1.] We say that $\Phi_q$ is defined by the counting of Maslov
index two holomorphic discs because of the following: Let $(p,v)\in
LX, x=\mu(p)\in P, L_x\subset X$ and
$\beta_1,\ldots,\beta_d\in\pi_2(\bar{X},L_x)$ be as before. For
$i=1,\ldots,d$, let $n_i(p)$ be the (algebraic) number of Maslov
index two holomorphic discs $\varphi:(D^2,\partial
D^2)\rightarrow(\bar{X},L_x)$ in the class $\beta_i$ whose boundary
passes through $p$. This number is well-defined since $\bar{X}$ is
toric Fano (see Section 3.1 and Section 4 in Auroux
\cite{Auroux07}). Then we can re-define
$$\Phi_q(p,v)=\sum_{\beta\in\pi_2^+(\bar{X},L_x),\ \partial\beta=v}
\frac{n_\beta(p)}{w(\beta)}e^{-\frac{1}{2\pi}\int_\beta\omega_{\bar{X}}},$$
where $n_\beta(p)=n_1(p)^{k_1}\ldots n_d(p)^{k_d}$ if
$\beta=\sum_{i=1}^d k_i\beta_i$. Defining $\Phi_q$ in this way makes
it explicit that $\Phi_q$ carries enumerative meaning. By
Theorem~\ref{cho-oh}, we have $n_i(p)=1$, for all $i=1,\ldots,d$ and
for any $p\in X$. So this definition reduces to the one above.
\item[2.] By definition, $\Phi_q$ is invariant under the
$T_N$-action on $X\subset\bar{X}$. Since
$X=T^*P/N=P\times\sqrt{-1}T_N$ (and the moment map $\mu:X\rightarrow
P$ is nothing but the projection to the first factor), we may view
$\Phi_q$ as a function on $P\times N$.
\item[3.] The function $\Phi_q$ is well-defined, i.e. the infinite
sum in its definition converges. To see this, notice that, by the
symplectic area formula (\ref{area}) of Cho-Oh, we have
$$\Phi_q(p,v)=\Bigg(\sum_{\substack{k_1,\ldots,k_d\in\mathbb{Z}_{\geq0},\\ \sum_{i=1}^d k_iv_i=v}}
\frac{q_1^{k_{n+1}}\ldots q_l^{k_d}}{k_1!\ldots
k_d!}\Bigg)e^{-\langle x,v\rangle},$$ and the sum inside the big
parentheses is less than $e^{n+q_1+\ldots+q_l}$.
\end{enumerate}
\end{nb}
For $T_N$-invariant functions $f,g:LX\rightarrow\mathbb{R}$, we
define their \textit{convolution product} $f\star
g:LX\rightarrow\mathbb{R}$ by
$$(f\star g)(p,v)=\sum_{v_1,v_2\in N,\ v_1+v_2=v}f(p,v_1)g(p,v_2),$$
for $(p,v)\in LX$. As in the theory of Fourier analysis, for the
convolution $f\star g$ to be well-defined, we need some conditions
for both $f$ and $g$. We leave this to Subsection~\ref{subsec3.2}
(see Definition~\ref{def3.3} and the subsequent discussion).
Nevertheless, if one of the functions is nonzero only for finitely
many $v\in N$, then the sum in the definition of $\star$ is a finite
sum, so it is well-defined. This is the case in the following
proposition.
\begin{prop}\label{prop2.1}[=part 1. of Proposition~\ref{prop1.1}]
The logarithmic derivatives of $\Phi_q$, with respect to $q_a$ for
$a=1,\ldots,l$, are given by
$$q_a\frac{\partial\Phi_q}{\partial q_a}=\Phi_q\star\Psi_{n+a}$$
where $\Psi_i:LX\rightarrow\mathbb{R}$ is defined, for
$i=1,\ldots,d$, by
$$\Psi_i(p,v)=\left\{ \begin{array}{ll}
                      e^{-\frac{1}{2\pi}\int_{\beta_i}\omega_{\bar{X}}} & \textrm{if $v=v_i$}\\
                      0                                   & \textrm{if $v\neq v_i$,}
                      \end{array}\right.$$
for $(p,v)\in LX=X\times N$, and with $x=\mu(p)\in P$,
$L_x=\mu^{-1}(x)$ and $\beta_1,\ldots,\beta_d\in\pi_2(\bar{X},L_x)$
as before.
\end{prop}
\begin{proof}
We will compute $q_l\frac{\partial\Phi_q}{\partial q_l}$. The others
are similar. By using Cho-Oh's formula (\ref{area}) and our choice
of the polytope $\bar{P}$, we have
$$e^{\langle x,v\rangle}\Phi_q(p,v)=
\sum_{\substack{k_1,\ldots,k_d\in\mathbb{Z}_{\geq0},\\ \sum_{i=1}^d
k_iv_i=v}}\frac{q_1^{k_{n+1}}\ldots q_l^{k_d}}{k_1!\ldots k_d!}.$$
Note that the right-hand-side is independent of $p\in X$.
Differentiating both sides with respect to $q_l$ gives
\begin{eqnarray*}
e^{\langle x,v\rangle}\frac{\partial\Phi_q(p,v)}{\partial q_l} & = &
\sum_{\substack{k_1,\ldots,k_{d-1}\in\mathbb{Z}_{\geq0},\
k_d\in\mathbb{Z}_{\geq1},\\ \sum_{i=1}^d
k_iv_i=v}}\frac{q_1^{k_{n+1}}\ldots
q_{l-1}^{k_{d-1}}q_l^{k_d-1}}{k_1!\ldots k_{d-1}!(k_d-1)!}\\
& = & \sum_{\substack{k_1,\ldots,k_d\in\mathbb{Z}_{\geq0},\\
\sum_{i=1}^d k_iv_i=v-v_d}}\frac{q_1^{k_{n+1}}\ldots
q_l^{k_d}}{k_1!\ldots
k_d!}\\
& = & e^{\langle x,v-v_d\rangle}\Phi_q(p,v-v_d).
\end{eqnarray*}
Hence, we obtain
\begin{eqnarray*}
q_l\frac{\partial\Phi_q(p,v)}{\partial q_l}=q_l e^{-\langle
x,v_d\rangle}\Phi_q(p,v-v_d).
\end{eqnarray*}
Now, by the definition of the convolution product $\star$, we have
\begin{eqnarray*}
\Phi_q\star\Psi_d(p,v)=\sum_{v_1,v_2\in N,\
v_1+v_2=v}\Phi_q(p,v_1)\Psi_d(p,v_2)=\Phi_q(p,v-v_d)\Psi_d(p,v_d),
\end{eqnarray*}
and
$\Psi_d(p,v_d)=e^{-\frac{1}{2\pi}\int_{\beta_d}\omega_{\bar{X}}}=e^{\lambda_d-\langle
x,v_d\rangle}=q_le^{-\langle x,v_d\rangle}$. The result follows.
\end{proof}
In the previous proposition, we introduce the $T_N$-invariant
functions $\Psi_1,\ldots,\Psi_d\in C^\infty(LX)$. Similar to what
has been said in Remark~\ref{rmk2.1}(1), these functions carry
enumerative meanings, and we should have defined $\Psi_i(p,v)$,
$i=1,\ldots,d$ in terms of the counting of Maslov index two
holomorphic discs in $(\bar{X},L_{\mu(p)})$ with boundary $v$ which
pass through $p$, i.e.
$$\Psi_i(p,v)=\left\{ \begin{array}{ll}
                      n_i(p)e^{-\frac{1}{2\pi}\int_{\beta_i}\omega_{\bar{X}}} & \textrm{if $v=v_i$}\\
                      0                                   & \textrm{if $v\neq v_i$,}
                      \end{array}\right.$$
for $(p,v)\in LX=X\times N$, where $x=\mu(p)\in P,
L_x=\mu^{-1}(x)\subset X$ and
$\beta_1,\ldots,\beta_d\in\pi_2(\bar{X},L_x)$ are as before. Again,
since the number $n_i(p)$ is always equal to one, for any $p\in X$
and for all $i=1,\ldots,d$, this definition of $\Psi_i$ is the same
as the previous one. But we should keep in mind that the function
$\Psi_i\in C^\infty(LX)$ encodes the following enumerative
information: for each $p\in X$, there is a unique Maslov index two
holomorphic disc $\varphi_i$ in the class $\beta_i$ with boundary in
the Lagrangian torus fiber $L_{\mu(p)}$ whose boundary passes
through $p$ and whose interior intersects the toric prime divisor
$D_i$ at one point. In view of this, we put the $d$ functions
$\{\Psi_i\}_{i=1}^d$, the $d$ families of Maslov index two
holomorphic discs $\{\varphi_i\}_{i=1}^d$ and the $d$ toric prime
divisors $\{D_i\}_{i=1}^d$ in one-to-one correspondences:
\begin{equation}\label{1-1}
\{\Psi_i\}_{i=1}^d\ \overset{1-1}{\longleftrightarrow}\
\{\varphi_i\}_{i=1}^d\ \overset{1-1}{\longleftrightarrow}\
\{D_i\}_{i=1}^d.
\end{equation}
Through these correspondences, we introduce linear relations in the
$d$-dimensional $\mathbb{C}$-vector space spanned by the functions
$\Psi_1,\ldots,\Psi_d$ using the linear equivalences among the
divisors $D_1,\ldots,D_d$.
\begin{defn}
Two linear combinations $\sum_{i=1}^d a_i\Psi_i$ and $\sum_{i=1}^d
b_i\Psi_i$, where $a_i,b_i\in\mathbb{C}$, are said to be linearly
equivalent, denoted by $\sum_{i=1}^d a_i\Psi_i\sim\sum_{i=1}^d
b_i\Psi_i$, if the corresponding divisors $\sum_{i=1}^d a_iD_i$ and
$\sum_{i=1}^d b_iD_i$ are linearly equivalent.
\end{defn}
We further define $\Psi_i^{-1}:LX\rightarrow\mathbb{R}$,
$i=1,\ldots,d$, by
$$\Psi_i^{-1}(p,v)=\left\{ \begin{array}{ll}
                      e^{\frac{1}{2\pi}\int_{\beta_i}\omega_X} & \textrm{if $v=-v_i$}\\
                      0                          & \textrm{if $v\neq-v_i$,}
                            \end{array}\right.$$
for $(p,v)\in LX$, so that $\Psi_i^{-1}\star\Psi_i=\mathbb{1}$,
where $\mathbb{1}:LX\rightarrow\mathbb{R}$ is the function defined
by
$$\mathbb{1}(p,v)=\left\{ \begin{array}{ll}
                            1 & \textrm{if $v=0$}\\
                            0 & \textrm{if $v\neq 0$.}
                          \end{array}\right.$$
The function $\mathbb{1}$ serves as a multiplicative identity for
the convolution product, i.e. $\mathbb{1}\star f=f\star\mathbb{1}=f$
for any $f\in C^\infty(LX)$. Now the second part of
Proposition~\ref{prop1.1} says that
\begin{prop}[=part 2. of Proposition~\ref{prop1.1}]\label{prop2.2}
We have a natural isomorphism of $\mathbb{C}$-algebras
\begin{equation}\label{isom2.3}
QH^*(\bar{X})\cong\mathbb{C}[\Psi_1^{\pm1},\ldots,\Psi_n^{\pm1}]/\mathcal{L},
\end{equation}
where $\mathbb{C}[\Psi_1^{\pm1},\ldots,\Psi_n^{\pm1}]$ is the
polynomial algebra generated by $\Psi_1^{\pm1},\ldots,\Psi_n^{\pm1}$
with respect to the convolution product $\star$ and $\mathcal{L}$ is
the ideal generated by linear equivalences, provided that $\bar{X}$
is a product of projective spaces.
\end{prop}
In the rest of this subsection, we will give an elementary proof of
this proposition by simple combinatorial arguments and computation
of certain Gromov-Witten invariants.

First of all, each toric prime divisor $D_i$ ($i=1,\ldots,d$)
determines a cohomology class in $H^2(\bar{X},\mathbb{C})$, which
will be, by abuse of notations, also denoted by $D_i$. It is known
by the general theory of toric varieties that the cohomology ring
$H^*(\bar{X},\mathbb{C})$ of the compact toric manifold $\bar{X}$ is
generated by the classes $D_1,\ldots,D_d$ in
$H^2(\bar{X},\mathbb{C})$ (see, for example, Fulton \cite{Fulton93}
or Audin \cite{Audin04}). More precisely, there is a presentation of
the form:
$$H^*(\bar{X},\mathbb{C})=\mathbb{C}[D_1,\ldots,D_d]/(\mathcal{L}+\mathcal{SR}),$$
where $\mathcal{L}$ is the ideal generated by linear equivalences
and $\mathcal{SR}$ is the \textit{Stanley-Reisner ideal} generated
by \textit{primitive relations} (see Batyrev \cite{Batyrev91}). Now,
by a result of Siebert and Tian (Proposition 2.2 in \cite{ST94}),
$QH^*(\bar{X})$ is also generated by $D_1,\ldots,D_d$ and a
presentation of $QH^*(\bar{X})$ is given by replacing each relation
in $\mathcal{SR}$ by its quantum counterpart. Denote by
$\mathcal{SR}_Q$ the quantum Stanley-Reisner ideal. Then we can
rephrase what we said as:
$$QH^*(\bar{X})=\mathbb{C}[D_1,\ldots,D_d]/(\mathcal{L}+\mathcal{SR}_Q).$$
The computation of $QH^*(\bar{X})$ (as a presentation) therefore
reduces to computing the generators of the ideal $\mathcal{SR}_Q$.

Let $\bar{X}=\mathbb{C}P^{n_1}\times\ldots\times\mathbb{C}P^{n_l}$
be a product of projective spaces. The complex dimension of
$\bar{X}$ is $n=n_1+\ldots+n_l$. For $a=1,\ldots,l$, let
$v_{1,a}=e_1,\ldots,v_{n_a,a}=e_{n_a},v_{n_a+1,a}=-\sum_{j=1}^{n_a}e_j\in
N_a$ be the primitive generators of the 1-dimensional cones in the
fan of $\mathbb{C}P^{n_a}$, where $\{e_1,\ldots,e_{n_a}\}$ is the
standard basis of $N_a\cong\mathbb{Z}^{n_a}$. For
$j=1,\ldots,n_a+1$, $a=1,\ldots,l$, we use the same symbol $v_{j,a}$
to denote the vector
$$(0,\ldots,\underbrace{v_{j,a}}_{\textrm{$a$-th}},\ldots,0)
\in N=N_1\oplus\ldots\oplus N_l,$$ where $v_{j,a}$ sits in the $a$th
place. These $d=\sum_{a=1}^l(n_a+1)=n+l$ vectors in $N$ are the
primitive generators of the 1-dimensional cones of the fan $\Sigma$
defining $\bar{X}$. In the following, we shall also denote the toric
prime divisor, the relative homotopy class, the family of Maslov
index two holomorphic discs with boundary in Lagrangian torus fibers
and the function on $LX$ corresponding to $v_{j,a}$ by $D_{j,a}$,
$\beta_{j,a}$, $\varphi_{j,a}$ and $\Psi_{j,a}$ respectively.
\begin{lem}
There are exactly $l$ primitive collections given by
$$\mathfrak{P}_a=\{v_{j,a}:j=1,\ldots,n_a+1\},\ a=1,\ldots,l,$$
and hence the Stanley-Reisner ideal of
$\bar{X}=\mathbb{C}P^{n_1}\times\ldots\times\mathbb{C}P^{n_l}$ is
given by
$$\mathcal{SR}=\langle D_{1,a}\cup\ldots\cup D_{n_a+1,a}:a=1\ldots,l\rangle.$$
\end{lem}
\begin{proof}
Let $\mathfrak{P}$ be any primitive collection. By definition,
$\mathfrak{P}$ is a collection of primitive generators of
1-dimensional cones of the fan $\Sigma$ defining $\bar{X}$ such that
for any $v\in\mathfrak{P}$, $\mathfrak{P}\setminus\{v\}$ generates a
$(|\mathfrak{P}|-1)$-dimensional cone in $\Sigma$, while
$\mathfrak{P}$ itself does not generate a
$|\mathfrak{P}|$-dimensional cone in $\Sigma$. Suppose that
$\mathfrak{P}\not\subset\mathfrak{P}_a$ for any $a$. For each $a$,
choose $v\in\mathfrak{P}\setminus(\mathfrak{P}\cap\mathfrak{P}_a)$.
By definition, $\mathfrak{P}\setminus\{v\}$ generates a cone in
$\Sigma$. But all the cones in $\Sigma$ are direct sums of cones in
the fans of the factors. So, in particular,
$\mathfrak{P}\cap\mathfrak{P}_a$, whenever it's nonempty, will
generate a cone in the fan of $\mathbb{C}P^{n_a}$. Since
$\mathfrak{P}=\bigsqcup_{a=1}^l\mathfrak{P}\cap\mathfrak{P}_a$, this
implies that the set $\mathfrak{P}$ itself generates a cone, which
is impossible. We therefore conclude that $\mathfrak{P}$ must be
contained in, and hence equal to one of the $\mathfrak{P}_a$'s.
\end{proof}
Hence, to compute the quantum Stanley-Reisner ideal
$\mathcal{SR}_Q$, we must compute the expression
$D_{1,a}\ast\ldots\ast D_{n_a+1,a}$, for $a=1,\ldots,l$, where
$\ast$ denotes the small quantum product of $QH^*(\bar{X})$. Before
doing this, we shall recall the definitions and properties of the
relevant Gromov-Witten invariants and the small quantum product for
$\bar{X}=\mathbb{C}P^{n_1}\times\ldots\times\mathbb{C}P^{n_l}$ as
follows.\\

For $\delta\in H_2(\bar{X},\mathbb{Z})$, let
$\overline{\mathcal{M}}_{0,m}(\bar{X},\delta)$ be the moduli space
of genus 0 stable maps with $m$ marked points and class $\delta$.
Since $\bar{X}$ is convex (i.e. for all maps
$\varphi:\mathbb{C}P^1\rightarrow\bar{X}$,
$H^1(\mathbb{C}P^1,\varphi^*T\bar{X})=0$), the moduli space
$\overline{\mathcal{M}}_{0,m}(\bar{X},\delta)$, if nonempty, is a
variety of pure complex dimension
$\textrm{dim}_\mathbb{C}(\bar{X})+c_1(\bar{X})\cdot\delta+m-3$ (see,
for example, the book \cite{Aluffi97}, p.3). For $k=1,\ldots,m$, let
$ev_k:\overline{\mathcal{M}}_{0,m}(\bar{X},\delta)\rightarrow\bar{X}$
be the evaluation map at the $k$th marked point, and let
$\pi:\overline{\mathcal{M}}_{0,m}(\bar{X},\delta)\rightarrow\overline{\mathcal{M}}_{0,m}$
be the forgetful map, where $\overline{\mathcal{M}}_{0,m}$ denotes
the Deligne-Mumford moduli space of genus 0 stable curves with $m$
marked points. Then, given cohomology classes $A\in
H^*(\overline{\mathcal{M}}_{0,m},\mathbb{Q})$ and
$\gamma_1,\ldots,\gamma_m\in H^*(\bar{X},\mathbb{Q})$, the
Gromov-Witten invariant is defined by
$$GW_{0,m}^{\bar{X},\delta}(A;\gamma_1,\ldots,\gamma_m)=
\int_{[\overline{\mathcal{M}}_{0,m}(\bar{X},\delta)]}\pi^*(A)\wedge
ev_1^*(\gamma_1)\wedge\ldots\wedge ev_m^*(\gamma_m),$$ where
$[\overline{\mathcal{M}}_{0,m}(\bar{X},\delta)]$ denotes the
fundamental class of $\overline{\mathcal{M}}_{0,m}(\bar{X},\delta)$.
Let $\ast$ be the small quantum product of $QH^*(\bar{X})$. Then it
is not hard to show that, for any classes
$\gamma_1,\ldots,\gamma_r\in H^*(\bar{X},\mathbb{Q})$, the
expression $\gamma_1\ast\ldots\ast\gamma_r$ can be computed by the
formula
$$\gamma_1\ast\ldots\ast\gamma_r=\sum_{\delta\in H_2(\bar{X},\mathbb{Z})}\sum_{i}
GW_{0,r+1}^{\bar{X},\delta}(\textrm{PD(pt)};\gamma_1,\ldots,\gamma_r,t_i)t^iq^\delta,$$
where $\{t_i\}$ is a basis of $H^*(\bar{X},\mathbb{Q})$, $\{t^i\}$
denotes the dual basis of $\{t_i\}$ with respect to the Poincar\'{e}
pairing, and $\textrm{PD(pt)}\in
H^{2m-6}(\overline{\mathcal{M}}_{0,m},\mathbb{Q})$ denotes the
Poincar\'{e} dual of a point in $\overline{\mathcal{M}}_{0,m}$ (see,
e.g. formula (1.4) in Spielberg \cite{Spielberg01}). Moreover, since
$\bar{X}$ is homogeneous of the form $G/P$, where $G$ is a Lie group
and $P$ is a parabolic subgroup, the Gromov-Witten invariants are
enumerative, in the sense that
$GW_{0,r+1}^{\bar{X},\delta}(\textrm{PD(pt)};\gamma_1,\ldots,\gamma_r,t_i)$
is equal to the number of holomorphic maps
$\varphi:(\mathbb{C}P^1;x_1,\ldots,x_r,x_{r+1})\rightarrow\bar{X}$
with $\varphi_*([\mathbb{C}P^1])=\delta$ such that
$(\mathbb{C}P^1;x_1,\ldots,x_r,x_{r+1})$ is a given point in
$\overline{\mathcal{M}}_{0,r+1}$, $\varphi(x_k)\in\Gamma_k$, for
$k=1,\ldots,r$, and $\varphi(x_{r+1})\in T_i$, where
$\Gamma_1,\ldots,\Gamma_r,T_i$ are representatives of cycles
Poincar\'{e} duals to the classes $\gamma_1,\ldots,\gamma_r,t_i$
respectively (see \cite{Aluffi97}, p.12).\\

We shall now use the above facts to compute $D_{1,a}\ast\ldots\ast
D_{n_a+1,a}$, which is given by the formula
$$D_{1,a}\ast\ldots\ast D_{n_a+1,a}=\sum_{\delta\in H_2(\bar{X},\mathbb{Z})}\sum_{i}
GW_{0,n_a+2}^{\bar{X},\delta}(\textrm{PD(pt)};D_{1,a},\ldots,D_{n_a+1,a},t_i)t^iq^\delta.$$
First of all, since $H_2(\bar{X},\mathbb{Z})$ is the kernel of the
boundary map
$\partial:\pi_2(\bar{X},L_x)=\mathbb{Z}^d\rightarrow\pi_1(L_x)=N$, a
homology class $\delta\in H_2(\bar{X},\mathbb{Z})$ can be
represented by a $d$-tuple of integers
$$\delta=(c_{1,1},\ldots,c_{n_1+1,1},\ldots,c_{1,b},\ldots,c_{n_b+1,b},\ldots,c_{1,l},
\ldots,c_{n_l+1,l})\in\mathbb{Z}^d$$ satisfying
$\sum_{b=1}^l\sum_{j=1}^{n_b+1}c_{j,b}v_{j,b}=0\in N$. Then we have
$c_1(\bar{X})\cdot\delta=\sum_{b=1}^l\sum_{j=1}^{n_b+1}c_{j,b}$. For
the Gromov-Witten invariant
$GW_{0,n_a+2}^{\bar{X},\delta}(\textrm{PD(pt)};D_{1,a},\ldots,D_{n_a+1,a},t_i)$
to be nonzero, $\delta$ must be represented by irreducible
holomorphic curves which pass through all the divisors
$D_{1,a},\ldots,D_{n_a+1,a}$. This implies that $c_{j,a}\geq1$, for
$j=1,\ldots,n_a+1$, and moreover, $\delta$ lies in the cone of
effective classes $H_2^{\textrm{eff}}(\bar{X},\mathbb{Z})\subset
H_2(\bar{X},\mathbb{Z})$. By Theorem 2.15 of Batyrev
\cite{Batyrev91}, $H_2^{\textrm{eff}}(\bar{X},\mathbb{Z})$ is given
by the kernel of the restriction of the boundary map
$\partial|_{\mathbb{Z}_{\geq0}^d}:\mathbb{Z}_{\geq0}^d\rightarrow
N$. So we must also have $c_{j,b}\geq0$ for all $j$ and $b$, and we
conclude that
$$c_1(\bar{X})\cdot\delta=\sum_{b=1}^l\sum_{j=1}^{n_b+1}c_{j,b}\geq n_a+1.$$
By dimension counting,
$GW_{0,n_a+2}^{\bar{X},\delta}(\textrm{PD(pt)};D_{1,a},\ldots,D_{n_a+1,a},t_i)\neq0$
only when
$$2(\textrm{dim}_\mathbb{C}(\bar{X})+c_1(\bar{X})\cdot\delta+(n_a+2)-3)=
2((n_a+2)-3)+2(n_a+1)+\textrm{deg}(t_i).$$ The above inequality then
implies that $\textrm{deg}(t_i)\geq2\textrm{dim}(\bar{X})$. We
therefore must have $t_i=\textrm{PD(pt)}\in
H^{2\textrm{dim}(\bar{X})}(\bar{X},\mathbb{Q})$ and $\delta\in
H_2(\bar{X},\mathbb{Z})$ is represented by the $d$-tuple of integers
$\delta_a:=(c_{1,1},\ldots,c_{n_l+1,l})\in\mathbb{Z}^d$, where
$$c_{j,b}=\left\{\begin{array}{ll}
                    1 & \textrm{if $b=a$ and $j=1,\ldots,n_a+1$}\\
                    0 & \textrm{otherwise,}
                 \end{array}\right.$$
i.e., $\delta=\delta_a$ is the pullback of the class of a line in
the factor $\mathbb{C}P^{n_a}$. Hence,
$$D_{1,a}\ast\ldots\ast D_{n_a+1,a}=
GW_{0,n_a+2}^{\bar{X},\delta_a}(\textrm{PD(pt)};D_{1,a},\ldots,D_{n_a+1,a},\textrm{PD(pt)})q^{\delta_a}.$$
By Theorem 9.3 in Batyrev~\cite{Batyrev93} (see also Siebert
\cite{Siebert95}, section 4), the Gromov-Witten invariant on the
right-hand-side is equal to 1. Geometrically, this means that, for
any given point $p\in X\subset\bar{X}$, there is a unique
holomorphic map
$\varphi_a:(\mathbb{C}P^1;x_1,\ldots,x_{n_a+2})\rightarrow\bar{X}$
with class $\delta_a$, $\varphi_a(x_j)\in D_{j,a}$, for
$j=1,\ldots,n_a+1$, $\varphi_a(x_{n_a+2})=p$ and such that
$(\mathbb{C}P^1;x_1,\ldots,x_{n_a+2})$ is a given configuration in
$\overline{\mathcal{M}}_{0,n_a+2}$. Also note that, for
$a=1,\ldots,l$,
$q^{\delta_a}=\exp(-\frac{1}{2\pi}\int_{\delta_a}\omega_{\bar{X}})=e^{-r_a}=q_a$,
where $(q_1,\ldots,q_l)$ are the coordinates of the K\"{a}hler cone
$\mathcal{K}(\bar{X})$. Thus, we have the following lemma.
\begin{lem}\label{lem2.2} For $a=1,\ldots,l$, we have
$$D_{1,a}\ast\ldots\ast D_{n_a+1,a}=q_a.$$
Hence, the quantum Stanley-Reisner ideal of
$\bar{X}=\mathbb{C}P^{n_1}\times\ldots\times\mathbb{C}P^{n_l}$ is
given by
$$\mathcal{SR}_Q=\langle D_{1,a}\ast\ldots\ast D_{n_a+1,a}-q_a:a=1\ldots,l\rangle,$$
and the quantum cohomology ring of $\bar{X}$ has a presentation
given by
\begin{eqnarray*}
QH^*(\bar{X})=&\frac{\mathbb{C}[D_{1,1},\ldots,D_{n_1+1,1},\ldots,D_{1,l},\ldots,D_{n_l+1,l}]}
{\big\langle D_{j,a}-D_{n_a+1,a}:j=1,\ldots,n_a,\
a=1,\ldots,l\big\rangle
+\big\langle\prod_{j=1}^{n_a+1}D_{j,a}-q_a:a=1,\ldots,l\big\rangle}.
\end{eqnarray*}
\end{lem}
Proposition~\ref{prop2.2} now follows from a simple combinatorial
argument:
\begin{proof}[Proof of Proposition~\ref{prop2.2}] For a general toric Fano manifold $\bar{X}$,
recall, from Remark~\ref{rmk2.0}, that we have chosen a
$\mathbb{Z}$-basis
$$Q_1=(Q_{11},\ldots,Q_{d1}),\ldots,Q_l=(Q_{1l},\ldots,Q_{dl})\in\mathbb{Z}^d$$
of $K=H_2(\bar{X},\mathbb{Z})$ such that $(Q_{n+a,b})_{1\leq a,b\leq
l}=\textrm{Id}_{l\times l}$. So, by Cho-Oh's symplectic area formula
(\ref{area}), for $a=1,\ldots,l$, we have
\begin{eqnarray*}
\Big(\sum_{i=1}^nQ_{ia}\int_{\beta_i}\omega_{\bar{X}}\Big)+\int_{\beta_{n+a}}\omega_{\bar{X}}
& = & 2\pi\Big(\sum_{i=1}^nQ_{ia}(\langle
x,v_i\rangle-\lambda_i)\Big)+2\pi(\langle x,v_{n+a}\rangle-\lambda_{n+a})\\
& = & 2\pi\langle
x,\sum_{i=1}^nQ_{ia}v_i+v_{n+a}\rangle-2\pi(\sum_{i=1}^aQ_{ia}\lambda_a+\lambda_{n+a})\\
& = & 2\pi r_a.
\end{eqnarray*}
Then, by the definition of the convolution product of functions on
$LX$, we have
\begin{eqnarray*}
\Psi_1^{Q_{1a}}\star\ldots\star\Psi_n^{Q_{na}}\star\Psi_{n+a}(x,v)
& = & \left\{
\begin{array}{ll}
         e^{-\frac{1}{2\pi}(\sum_{i=1}^nQ_{ia}\int_{\beta_i}\omega_{\bar{X}})
         -\frac{1}{2\pi}\int_{\beta_{n+a}}\omega_{\bar{X}}} & \textrm{if $v=0$}\\
         0 & \textrm{if $v\neq0$}
         \end{array} \right.\\
& = & \left\{ \begin{array}{ll}
         e^{-r_a} & \textrm{if $v=0$}\\
         0 & \textrm{if $v\neq0$}
         \end{array} \right.\\
& = & q_a\mathbb{1},
\end{eqnarray*}
or
$\Psi_{n+a}=q_a(\Psi_1^{-1})^{Q_{1a}}\star\ldots\star(\Psi_n^{-1})^{Q_{na}}$,
for $a=1,\ldots,l$.

Suppose that the following condition: $Q_{ia}\geq0$ for
$i=1,\ldots,n$, $a=1,\ldots,l$, and for each $i=1,\ldots,n$, there
exists $1\leq a\leq l$ such that $Q_{ia}>0$, is satisfied, which is
the case when $\bar{X}$ is a product of projective spaces. Then the
inclusion
$$\mathbb{C}[\Psi_1,\ldots,\Psi_n,\Psi_{n+1},\ldots,\Psi_d]\hookrightarrow
\mathbb{C}[\Psi_1^{\pm1},\ldots,\Psi_n^{\pm1}]$$ is an isomorphism.
Consider the surjective map
$$\rho:\mathbb{C}[D_1,\ldots,D_d]\rightarrow\mathbb{C}[\Psi_1,\ldots,\Psi_d]$$
defined by mapping $D_i$ to $\Psi_i$ for $i=1,\ldots,d$. This map is
not injective because there are nontrivial relations in
$\mathbb{C}[\Psi_1,\ldots,\Psi_d]$ generated by the relations
$$\Psi_1^{Q_{1a}}\star\ldots\star\Psi_n^{Q_{na}}\star\Psi_{n+a}-q_a\mathbb{1}=0,\ a=1,\ldots,l.$$
By Lemma~\ref{lem2.2}, the kernel of $\rho$ is exactly given by the
ideal $\mathcal{SR}_Q$ when $\bar{X}$ is a product of projective
spaces. Thus, we have an isomorphism
$$\mathbb{C}[D_1,\ldots,D_d]/\mathcal{SR}_Q\overset{\cong}{\longrightarrow}\mathbb{C}[\Psi_1,\ldots,\Psi_d].$$
Since
$(\mathbb{C}[D_1,\ldots,D_d]/\mathcal{SR}_Q)/\mathcal{L}=\mathbb{C}[D_1,\ldots,D_d]/(\mathcal{L}+\mathcal{SR}_Q)
=QH^*(\bar{X})$, we obtain the desired isomorphism
$$QH^*(\bar{X})\cong\mathbb{C}[\Psi_1,\ldots,\Psi_d]/\mathcal{L}\cong
\mathbb{C}[\Psi_1^{\pm1},\ldots,\Psi_n^{\pm1}]/\mathcal{L},$$
provided that $\bar{X}$ is a product of projective spaces.
\end{proof}
\begin{nb}\label{rmk2.3}$\mbox{}$
\begin{enumerate}
\item[1.] In \cite{Batyrev93}, Theorem 5.3, Batyrev gave a "formula" for the
quantum Stanley-Reisner ideal $\mathcal{SR}_Q$ for any compact toric
K\"{a}hler manifolds, using his own definition of the small quantum
product (which is different from the usual one because Batyrev
counted only holomorphic maps from $\mathbb{C}P^1$). By Givental's
mirror theorem~\cite{Givental97a}, Batyrev's formula is true, using
the usual definition of the small quantum product, for all toric
Fano manifolds. Our proof of Lemma~\ref{lem2.2} is nothing but a
simple verification of Batyrev's formula in the case of products of
projective spaces, without using Givental's mirror theorem.
\item[2.] In any event,
Batyrev's formula in \cite{Batyrev93} for a presentation of the
small quantum cohomology ring $QH^*(\bar{X})$ of a toric Fano
manifold $\bar{X}$ is correct. In the same paper, Batyrev also
proved that $QH^*(\bar{X})$ is canonically isomorphic to the
Jacobian ring $Jac(W)$, where $W$ is the superpotential mirror to
$\bar{X}$ (Theorem 8.4 in \cite{Batyrev93}). Now, by
Theorem~\ref{thm3.3} in Subsection~\ref{subsec3.3}, the inverse SYZ
transformation $\mathcal{F}^{-1}$ gives a canonical isomorphism
$\mathcal{F}^{-1}:Jac(W)\overset{\cong}{\rightarrow}
\mathbb{C}[\Psi_1^{\pm1},\ldots,\Psi_n^{\pm1}]/\mathcal{L}$. Then,
the composition map
$QH^*(\bar{X})\rightarrow\mathbb{C}[\Psi_1^{\pm1},\ldots,\Psi_n^{\pm1}]/\mathcal{L}$,
which maps $D_i$ to $\Psi_i$, for $i=1,\ldots,d$, is an isomorphism.
This proves Proposition~\ref{prop2.2} all toric Fano manifolds. We
choose not to use this proof because all the geometry is then hid by
the use of Givental's mirror theorem.
\end{enumerate}
\end{nb}

\subsection{The role of tropical geometry}\label{subsec2.2}

While our proof of the isomorphism (\ref{isom2.3}) in
Proposition~\ref{prop2.2} is combinatorial in nature, the best way
to understand the geometry behind it is through the correspondence
between holomorphic curves and discs in $\bar{X}$ and their tropical
counterparts in $N_\mathbb{R}$. Indeed, this is the main reason why
we confine ourselves to the case of products of projective spaces.
Our first task is to define a tropical analog $QH^*_{trop}(\bar{X})$
of the small quantum cohomology ring of $\bar{X}$, when $\bar{X}$ is
a product of projective spaces. For this, we shall recall some
notions in tropical geometry. We will not state the precise
definitions, for which we refer the reader to Mikhalkin
\cite{Mikhalkin03}, \cite{Mikhalkin06}, \cite{Mikhalkin07} and
Nishinou-Siebert \cite{NS04}.\\

A \textit{genus 0 tropical curve with $m$ marked points} is a
connected tree $\Gamma$ with exactly $m$ unbounded edges (also
called leaves) and each bounded edge is assigned a positive length.
Let $\overline{\mathcal{M}}_{0,m}^{trop}$ be the moduli space of
genus 0 tropical curves with $m$ marked points (modulo
isomorphisms). The combinatorial types of $\Gamma$ partition
$\overline{\mathcal{M}}_{0,m}^{trop}$ into disjoint subsets, each of
which has the structure of a polyhedral cone $\mathbb{R}_{>0}^e$
(where $e$ is the number of bounded edges in $\Gamma$). There is a
distinguished point in $\overline{\mathcal{M}}_{0,m}^{trop}$
corresponding to the (unique) tree $\Gamma_m$ with exactly one
($m$-valent) vertex $V$, $m$ unbounded edges $E_1,\ldots,E_m$ and
\textit{no} bounded edges. See Figure 2.1 below. We will fix this
point in $\overline{\mathcal{M}}_{0,m}^{trop}$; this is analog to
fixing a point in $\overline{\mathcal{M}}_{0,m}$.
\begin{figure}[ht]
\setlength{\unitlength}{1mm}
\begin{picture}(100,30)
\curve(30,11.5, 12,27.5) \put(15,25.5){$E_2$} \curve(30,11.5,
12,-1.5) \put(16.5,-1){$E_3$} \curve(30,11.5, 40,30.5)
\put(39,26.5){$E_1$} \curve(30,11.5, 48,0.5) \put(41,-0.2){$E_4$}
\put(29.2,10.7){$\bullet$} \put(32,11){$V$} \put(55,11.5){Figure
2.1: $\Gamma_4\in\overline{\mathcal{M}}_{0,4}^{trop}.$}
\end{picture}
\end{figure}

Let $\Sigma$ be the fan defining
$\bar{X}=\mathbb{C}P^{n_1}\times\ldots\times\mathbb{C}P^{n_l}$, and
denote by
$\Sigma[1]=\{v_{1,1},\ldots,v_{n_1+1,1},\ldots,v_{1,a},\ldots,v_{n_a+1,a},\ldots,
v_{1,l},\ldots,v_{n_l+1,l}\}\subset N$ the set of primitive
generators of 1-dimensional cones in $\Sigma$. Let
$h:\Gamma_m\rightarrow N_\mathbb{R}$ be a continuous embedding such
that, for each $k=1,\ldots,m$,
$h(E_k)=h(V)+\mathbb{R}_{\geq0}v(E_k)$ for some
$v(E_k)\in\Sigma[1]$, and the following \textit{balancing condition}
is satisfied:
$$\sum_{k=1}^m v(E_k)=0.$$
Then the tuple $(\Gamma_m;E_1,\ldots,E_m;h)$ is a
\textit{parameterized $m$-marked, genus 0 tropical curve in
$\bar{X}$}. The \textit{degree} of $(\Gamma_m;E_1,\ldots,E_m;h)$ is
the $d$-tuple of integers
$\delta(h)=(c_{1,1},\ldots,c_{n_1+1,1},\ldots,c_{1,a},\ldots,c_{n_a+1,a},\ldots,
c_{1,l},\ldots,c_{n_l+1,l})\in\mathbb{Z}^d$, where
$$c_{j,a}=\left\{\begin{array}{ll}
                    1 & \textrm{if $v_{j,a}\in\{v(E_1),\ldots,v(E_m)\}$}\\
                    0 & \textrm{otherwise.}
                 \end{array}\right.$$
By the balancing condition, we have
$\sum_{a=1}^l\sum_{j=1}^{n_a+1}c_{j,a}v_{j,a}=0$, i.e. $\delta(h)$
lies in the kernel of $\partial:\mathbb{Z}^d\rightarrow N$, and so
$\delta(h)\in H_2(\bar{X},\mathbb{Z})$.

We want to consider the tropical counterpart, denoted by
$$TGW_{0,n_a+2}^{\bar{X},\delta}(\textrm{PD(pt)};D_{1,a},\ldots,D_{n_a+1,a},t_i),$$
of the Gromov-Witten invariant
$GW_{0,n_a+2}^{\bar{X},\delta}(\textrm{PD(pt)};D_{1,a},\ldots,D_{n_a+1,a},t_i)$.
\footnote{"TGW" stands for "tropical Gromov-Witten".} Since a
general definition is not available, we introduce a tentative
definition as follows.
\begin{defn}
We define
$TGW_{0,n_a+2}^{\bar{X},\delta}(\textrm{PD(pt)};D_{1,a},\ldots,D_{n_a+1,a},t_i)$
to be the number of parameterized $(n_a+1)$-marked, genus 0 tropical
curves of the form $(\Gamma_{n_a+1};E_1,\ldots,E_{n_a+1};h)$ with
$\delta(h)=\delta$ such that
$h(E_j)=h(V)+\mathbb{R}_{\geq0}v_{j,a}$, for $j=1,\ldots,n_a+1$, and
$h(V)\in\textrm{Log}(T_i)$, where $T_i$ is a cycle Poincar\'{e} dual
to $t_i$, whenever this number is finite. We set
$TGW_{0,n_a+2}^{\bar{X},\delta}(\textrm{PD(pt)};D_{1,a},\ldots,D_{n_a+1,a},t_i)$
to be 0 if this number is infinite. Here, $\textrm{Log}:X\rightarrow
N_\mathbb{R}$ is the map, after identifying $X$ with
$(\mathbb{C}^*)^n$, defined by
$\textrm{Log}(w_1,\ldots,w_n)=(\log|w_1|,\ldots,\log|w_n|)$, for
$(w_1,\ldots,w_n)\in X$.
\end{defn}
We then define the \textit{tropical small quantum cohomology ring}
$QH^*_{trop}(\bar{X})$ of
$\bar{X}=\mathbb{C}P^{n_1}\times\ldots\times\mathbb{C}P^{n_l}$ as a
presentation:
$$QH^*_{trop}(\bar{X})=\mathbb{C}[D_{1,1},\ldots,D_{n_1+1,1},\ldots,D_{1,l}\ldots,D_{n_l+1,l}]/
(\mathcal{L}+\mathcal{SR}_Q^{trop}),$$ where $\mathcal{SR}_Q^{trop}$
is the tropical version of the quantum Stanley-Reisner ideal,
defined to be the ideal generated by the relations
$$D_{1,a}\ast_T\ldots\ast_T D_{n_a+1,a}=\sum_{\delta\in H_2(\bar{X},\mathbb{Z})}\sum_{i}
TGW_{0,n_a+2}^{\bar{X},\delta}(\textrm{PD(pt)};D_{1,a},\ldots,D_{n_a+1,a},t_i)t^iq^\delta,$$
for $a=1,\ldots,l$. Here $\ast_T$ denotes the product in
$QH^*_{trop}(\bar{X})$, which we call the tropical small quantum
product. It is not hard to see that, as in the holomorphic case, we
have
$$TGW_{0,n_a+2}^{\bar{X},\delta}(\textrm{PD(pt)};D_{1,a},\ldots,D_{n_a+1,a},t_i)
=\left\{\begin{array}{ll}
            1 & \textrm{if $t_i=\textrm{PD(pt)}$ and $\delta=\delta_a$}\\
            0 & \textrm{otherwise.}
        \end{array}\right.$$
Indeed, as a special case of the \textit{correspondence theorem} of
Mikhalkin \cite{Mikhalkin03} and Nishinou-Siebert \cite{NS04}, we
have: For a given point $p\in X$, let $\xi:=\textrm{Log}(p)\in
N_\mathbb{R}$. Then the unique holomorphic curve
$\varphi_a:(\mathbb{C}P^1;x_1,\ldots,x_{n_a+2})\rightarrow\bar{X}$
with class $\delta_a$, $\varphi_a(x_j)\in D_{j,a}$, for
$j=1,\ldots,n_a+1$, $\varphi_a(x_{n_a+2})=p$ and such that
$(\mathbb{C}P^1;x_1,\ldots,x_{n_a+2})$ is a given configuration in
$\overline{\mathcal{M}}_{0,n_a+2}$, is corresponding to the unique
parameterized $(n_a+1)$-marked tropical curve
$(\Gamma_{n_a+1};E_1,\ldots,E_{n_a+1};h_a)$ of genus 0 and degree
$\delta_a$ such that $h_a(V)=\xi$ and
$h_a(E_j)=\xi+\mathbb{R}_{\geq0}v_{j,a}$, for $j=1,\ldots,n_a+1$. It
follows that
$$\mathcal{SR}_Q^{trop}=\langle D_{1,a}\ast_T\ldots\ast_T D_{n_a+1,a}-q_a:a=1\ldots,l\rangle,$$
and there is a canonical isomorphism
\begin{equation}\label{step1}
QH^*(\bar{X})\cong QH^*_{trop}(\bar{X}).
\end{equation}
\begin{nb}
All these arguments and definitions rely, in an essential way, on
the fact that $\bar{X}$ is a product of projective spaces, so that
Gromov-Witten invariants are enumerative and all the (irreducible)
holomorphic curves, which contribute to $QH^*(\bar{X})$, are not
mapped into the toric boundary divisor $D_\infty$. Remember that
tropical geometry cannot be used to count nodal curves or curves
with irreducible components mapping into $D_\infty$.
\end{nb}
Next, we take a look at tropical discs. Consider the point
$\Gamma_1\in\overline{\mathcal{M}}_{0,1}^{trop}$. This is nothing
but a half line, consisting of an unbounded edge $E$ emanating from
a univalent vertex $V$. See Figure 2.2 below.
\begin{figure}[ht]
\setlength{\unitlength}{1mm}
\begin{picture}(100,20)
\curve(20,3, 63,20) \put(19,2){$\bullet$} \put(16,1){$V$}
\put(41,7){$E$} \put(64,10){Figure 2.2:
$\Gamma_1\in\overline{\mathcal{M}}_{0,1}^{trop}.$}
\end{picture}
\end{figure}

\noindent A \textit{parameterized Maslov index two tropical disc in
$\bar{X}$} is a tuple $(\Gamma_1,E,h)$, where $h:\Gamma_1\rightarrow
N_\mathbb{R}$ is an embedding such that
$h(E)=h(V)+\mathbb{R}_{\geq0}v$ for some
$v\in\Sigma[1]$.\footnote{For precise definitions of general
tropical discs (with higher Maslov indices), we refer the reader to
Nishinou \cite{Nishinou06}; see also the recent work of Gross
\cite{Gross09}.} For any given point $\xi\in N_\mathbb{R}$, it is
obvious that, there is a unique parameterized Maslov index two
tropical disc $(\Gamma_1,E,h_{j,a})$ such that $h_{j,a}(V)=\xi$ and
$h_{j,a}(E)=\xi+\mathbb{R}_{\geq0}v_{j,a}$, for any
$v_{j,a}\in\Sigma[1]$. Comparing this to the result
(Theorem~\ref{cho-oh}) of Cho-Oh on the classification of Maslov
index two holomorphic discs in $\bar{X}$ with boundary in the
Lagrangian torus fiber $L_\xi:=\textrm{Log}^{-1}(\xi)\subset X$, we
get a one-to-one correspondence between the families of Maslov index
two holomorphic discs in $(\bar{X},L_\xi)$ and the parameterized
Maslov index two tropical discs $(\Gamma_1,E,h)$ in $\bar{X}$ such
that $h(V)=\xi$. We have the holomorphic disc
$\varphi_{j,a}:(D^2,\partial D^2)\rightarrow(\bar{X},L_\xi)$
corresponding to the tropical disc
$(\Gamma_1,E,h_{j,a})$.\footnote{This correspondence also holds for
other toric manifolds, not just for products of projective spaces.}
Then, by (\ref{1-1}), we also get a one-to-one correspondence
between the parameterized Maslov index two tropical discs
$(\Gamma_1,E,h_{j,a})$ in $\bar{X}$ and the functions
$\Psi_{j,a}:L_X\rightarrow\mathbb{R}$:
\begin{equation*}
\{\varphi_{j,a}\}\ \overset{1-1}{\longleftrightarrow}\
\{(\Gamma_1,E,h_{j,a})\}\ \overset{1-1}{\longleftrightarrow}\
\{\Psi_{j,a}\}.
\end{equation*}

Now, while the canonical isomorphism
\begin{equation}\label{step2}
QH^*_{trop}(\bar{X})\cong\mathbb{C}[\Psi_{1,1}^{\pm1},\ldots,\Psi_{n_1,1}^{\pm1},\ldots,
\Psi_{1,l}^{\pm1},\ldots,\Psi_{n_l,l}^{\pm1}]/\mathcal{L}
\end{equation}
follows from the same simple combinatorial argument in the proof of
Proposition~\ref{prop2.2}, the geometry underlying it is exhibited
by a simple but crucial observation, which we formulate as the
following proposition.
\begin{prop}\label{prop2.3}
Let $\xi\in N_\mathbb{R}$, then the unique parameterized
$(n_a+1)$-marked, genus 0 tropical curve
$(\Gamma_{n_a+1};E_1,\ldots,E_{n_a+1};h_a)$ such that $h_a(V)=\xi$
and $h_a(E_j)=\xi+\mathbb{R}_{\geq0}v_{j,a}$, for
$j=1,\ldots,n_a+1$, is obtained by gluing the $n_a+1$ parameterized
Maslov index two tropical discs
$(\Gamma_1,E,h_{1,a}),\ldots,(\Gamma_1,E,h_{n_a+1,a})$ with
$h_{j,a}(V)=\xi$, for $j=1,\ldots,n_a+1$, in the following sense:
The map $h:(\Gamma_{n_a+1};E_1,\ldots,E_{n_a+1})\rightarrow
N_\mathbb{R}$ defined by $h|_{E_j}=h_{j,a}|_E$, for
$j=1,\ldots,n_a+1$, gives a parameterized $(n_a+1)$-marked, genus 0
tropical curve, which coincides with
$(\Gamma_{n_a+1};E_1,\ldots,E_{n_a+1};h_a)$.
\end{prop}
\begin{proof}
Since $\sum_{j=1}^{n_a+1}v_{j,a}=0$, the balancing condition at
$V\in\Gamma_{n_a+1}$ is automatically satisfied. So $h$ defines a
parameterized $(n_a+1)$-marked, genus 0 tropical curve, which
satisfies the same conditions as
$(\Gamma_{n_a+1};E_1,\ldots,E_{n_a+1};h_a)$.
\end{proof}
For example, in the case of $\bar{X}=\mathbb{C}P^2$, this can be
seen in Figure 2.3 below.
\begin{figure}[ht]
\setlength{\unitlength}{1mm}
\begin{picture}(100,35)
\curve(20,17, 38,17) \curve(20,17, 20,35) \curve(20,17, 5,2)
\put(10,1){$(\Gamma_3,h)$} \put(19.1,16.1){$\bullet$}
\put(16,16){$V$} \put(42,16){glued from} \curve(74,17, 93,17)
\put(73,16.1){$\bullet$} \put(83,13.5){$(\Gamma_1,h_1)$}
\curve(72,19, 72,37) \put(71.1,18){$\bullet$}
\put(72.5,29){$(\Gamma_1,h_2)$} \curve(71,16, 56,1)
\put(70,15){$\bullet$} \put(60,2){$(\Gamma_1,h_3)$}
\put(35,-2){Figure 2.3}
\end{picture}
\end{figure}

The functions $\Psi_{j,a}$'s could have been defined by counting
parameterized Maslov index two tropical discs, instead of counting
Maslov index two holomorphic discs. So the above proposition indeed
gives a geometric reason to explain why the relation
$$D_{1,a}\ast_T\ldots\ast_T D_{n_a+1,a}=q_a$$
in $QH^*_{trop}(\bar{X})$ should coincide with the relation
$$\Psi_{1,a}\star\ldots\star\Psi_{n_a+1,a}=q_a\mathbb{1}$$
in $\mathbb{C}[\Psi_{1,1}^{\pm1},\ldots,\Psi_{n_1,1}^{\pm1},\ldots,
\Psi_{1,l}^{\pm1},\ldots,\Psi_{n_l,l}^{\pm1}]/\mathcal{L}$. The
convolution product $\star$ may then be thought of as a way to
encode the gluing of tropical discs.

We summarize what we have said as follows: In the case of products
of projective spaces, we factor the isomorphism
$QH^*(\bar{X})\cong\mathbb{C}[\Psi_1^{\pm1},\ldots,\Psi_n^{\pm1}]/\mathcal{L}$
in Proposition~\ref{prop2.2} into two isomorphisms (\ref{step1}) and
(\ref{step2}). The first one comes from the correspondence between
holomorphic curves in $\bar{X}$ which contribute to $QH^*(\bar{X})$
and tropical curves in $N_\mathbb{R}$ which contribute to
$QH^*_{trop}(\bar{X})$. The second isomorphism is due to, on the one
hand, the fact that each tropical curve which contributes to
$QH^*_{trop}(\bar{X})$ can be obtained by gluing Maslov index two
tropical discs, and, on the other hand, the correspondence between
these tropical discs in $N_\mathbb{R}$ and Maslov index two
holomorphic discs in $\bar{X}$ with boundary on Lagrangian torus
fibers. See Figure 2.4 below.
\begin{figure}[ht]
\setlength{\unitlength}{1mm}
\begin{picture}(100,31)
\put(5,30){$QH^*(\bar{X})$} \put(11,16){\vector(0,1){12.5}}
\put(11,16){\vector(0,-1){11}} \put(4,1){$QH^*_{trop}(\bar{X})$}
\put(28.5,26){\vector(2,-1){13}} \put(28.5,26){\vector(-2,1){9}}
\put(31,25.5){Prop~\ref{prop2.2}} \put(28.5,6){\vector(2,1){13}}
\put(28.5,6){\vector(-2,-1){7}} \put(31,4){Prop~\ref{prop2.3}}
\put(41,14.5){$\mathbb{C}[\Psi_1^{\pm1},\ldots,\Psi_n^{\pm1}]/\mathcal{L}$}
\put(80,15.5){\vector(1,0){9}} \put(80,15.5){\vector(-1,0){8}}
\put(80,16.5){$\mathcal{F}$} \put(90,14.5){$Jac(W)$}
\put(50,-2){Figure 2.4}
\end{picture}
\end{figure}

\noindent Here $\mathcal{F}$ denotes the SYZ mirror transformation
for $\bar{X}$, which is the subject of Section~\ref{sec3}.

\section{SYZ mirror transformations}\label{sec3}

In this section, we first derive Hori-Vafa's mirror manifold using
semi-flat SYZ mirror transformations. Then we introduce the main
character in this paper: the SYZ mirror transformations for toric
Fano manifolds, and prove our main result.

\subsection{Derivation of Hori-Vafa's mirror manifold by T-duality}\label{subsec3.1}

Recall that we have an exact sequence (\ref{seq2.1}):
\begin{equation}\label{seq3.1}
0\longrightarrow K\overset{\iota}{\longrightarrow}\mathbb{Z}^d
\overset{\partial}{\longrightarrow}N\longrightarrow0,
\end{equation}
and we denote by
$$Q_1=(Q_{11},\ldots,Q_{d1}),\ldots,Q_l=(Q_{1l},\ldots,Q_{dl})\in\mathbb{Z}^d$$
a $\mathbb{Z}$-basis of $K$. The mirror manifold of $\bar{X}$,
derived by Hori and Vafa in \cite{HV00} using physical arguments, is
the complex submanifold
$$Y_{HV}=\Big\{(Z_1,\ldots,Z_d)\in\mathbb{C}^d:\prod_{i=1}^d Z_i^{Q_{ia}}=q_a,\ a=1,\ldots,l\Big\},$$
in $\mathbb{C}^d$, where $q_a=e^{-r_a}=\exp(-\sum_{i=1}^d
Q_{ia}\lambda_i)$, for $a=1,\ldots,l$. As a complex manifold,
$Y_{HV}$ is biholomorphic to the algebraic torus $(\mathbb{C}^*)^n$.
By our choice of the $\mathbb{Z}$-basis $Q_1,\ldots,Q_l$ of $K$ in
Remark~\ref{rmk2.0}, $Y_{HV}$ can also be written as
$$Y_{HV}=\Big\{(Z_1,\ldots,Z_d)\in\mathbb{C}^d:Z_1^{Q_{1a}}\ldots Z_n^{Q_{na}}Z_{n+a}=q_a,\ a=1,\ldots,l\Big\}.$$
Note that, in fact, $Y_{HV}\subset(\mathbb{C}^*)^d$. In terms of
these coordinates, Hori and Vafa predicted that the superpotential
$W:Y_{HV}\rightarrow\mathbb{C}$ is given by
\begin{eqnarray*}
W & = & Z_1+\ldots+Z_d\\
  & = & Z_1+\ldots+Z_n+\frac{q_1}{Z_1^{Q_{11}}\ldots Z_n^{Q_{n1}}}+
\ldots+\frac{q_l}{Z_1^{Q_{1l}}\ldots Z_n^{Q_{nl}}}.
\end{eqnarray*}
The goal of this subsection is to show that the SYZ mirror manifold
$Y_{SYZ}$, which is obtained by applying T-duality to the open dense
orbit $X\subset\bar{X}$, is contained in Hori-Vafa's manifold
$Y_{HV}$ as a bounded open subset. The result itself is not new, and
can be found, for example, in Auroux \cite{Auroux07}, Proposition
4.2. For the sake of completeness, we give a self-contained proof,
which will show how T-duality, i.e. fiberwise dualizing torus
bundles, transforms the \textit{symplectic quotient space} $X$ into
the \textit{complex subspace} $Y_{SYZ}$.\\

We shall first briefly recall the constructions of $\bar{X}$ and $X$
as symplectic quotients. For more details, we refer the reader to
Appendix 1 in Guillemin \cite{Guillemin94}.

From the above exact sequence (\ref{seq3.1}), we get an exact
sequence of real tori
\begin{equation}\label{seq3.2}
1\longrightarrow T_K\overset{\iota}{\longrightarrow}T^d
\overset{\partial}{\longrightarrow}T_N\longrightarrow1,
\end{equation}
where $T^d=\mathbb{R}^d/(2\pi\mathbb{Z})^d$ and we denote by
$K_\mathbb{R}$ and $T_K$ the real vector space
$K\otimes_\mathbb{Z}\mathbb{R}$ and the torus $K_\mathbb{R}/K$
respectively. Considering their Lie algebras and dualizing give
another exact sequence
\begin{equation}\label{seq3.2'}
0\longrightarrow
M_\mathbb{R}\overset{\check{\partial}}{\longrightarrow}(\mathbb{R}^d)^\vee
\overset{\check{\iota}}{\longrightarrow}K_\mathbb{R}^\vee\longrightarrow0.
\end{equation}
Denote by $W_1,\ldots,W_d\in\mathbb{C}$ the complex coordinates on
$\mathbb{C}^d$. The standard diagonal action of $T^d$ on
$\mathbb{C}^d$ is Hamiltonian with respect to the standard
symplectic form $\frac{\sqrt{-1}}{2}\sum_{i=1}^d dW_i\wedge
d\bar{W}_i$ and the moment map
$h:\mathbb{C}^d\rightarrow(\mathbb{R}^d)^\vee$ is given by
$$h(W_1,\ldots,W_d)=\frac{1}{2}(|W_1|^2,\ldots,|W_d|^2).$$
Restricting to $T_K$, we get a Hamiltonian action of $T_K$ on
$\mathbb{C}^d$ with moment map $h_K=\check{\iota}\circ
h:\mathbb{C}^d\rightarrow\check{K}_\mathbb{R}$. In terms of the
$\mathbb{Z}$-basis $\{Q_1,\ldots,Q_l\}$ of $K$, the map
$\check{\iota}:(\mathbb{R}^d)^\vee\rightarrow K^\vee_\mathbb{R}$ is
given by
\begin{equation}\label{iota*}
\check{\iota}(X_1,\ldots,X_d)=\Bigg(\sum_{i=1}^d
Q_{i1}X_i,\ldots,\sum_{i=1}^d Q_{il}X_i\Bigg),
\end{equation}
for $(X_1,\ldots,X_d)\in(\mathbb{R}^d)^\vee$, in the coordinates
associated to the dual basis $\check{Q}_1,\ldots,\check{Q}_l$ of
$K^\vee=\textrm{Hom}(K,\mathbb{Z})$. The moment map
$h_K:\mathbb{C}^d\rightarrow K^\vee_\mathbb{R}$ can thus be written
as
$$h_K(W_1,\ldots,W_d)=\frac{1}{2}\Bigg(\sum_{i=1}^d
Q_{i1}|W_i|^2,\ldots,\sum_{i=1}^d Q_{il}|W_i|^2\Bigg)\in
K^\vee_\mathbb{R}.$$ In these coordinates,
$r=(r_1,\ldots,r_l)=-\check{\iota}(\lambda_1,\ldots,\lambda_d)$ is
an element in $K^\vee_\mathbb{R}=H^2(\bar{X},\mathbb{R})$, and
$\bar{X}$ and $X$ are given by the symplectic quotients
$$\bar{X}=h_K^{-1}(r)/T_K\textrm{ and }X=(h_K^{-1}(r)\cap(\mathbb{C}^*)^d)/T_K$$
respectively.

In the above process, the image of $h_K^{-1}(r)$ under the map
$h:\mathbb{C}^d\rightarrow(\mathbb{R}^d)^\vee$ lies inside the
affine linear subspace
$M_\mathbb{R}(r)=\{(X_1,\ldots,X_d)\in(\mathbb{R}^d)^\vee:\check{\iota}(X_1,\ldots,X_d)=r\}$,
i.e. a translate of $M_\mathbb{R}$. In fact,
$h(h_K^{-1}(r))=\check{\iota}^{-1}(r)\cap\{(X_1,\ldots,X_d)\in(\mathbb{R}^d)^\vee:X_i\geq0,\textrm{
for $i=1,\ldots,d$}\}$ is the polytope $\bar{P}\subset
M_\mathbb{R}(r)$, and $h(h_K^{-1}(r)\cap(\mathbb{C}^*)^d)=
\check{\iota}^{-1}(r)\cap\{(X_1,\ldots,X_d)\in(\mathbb{R}^d)^\vee:X_i>0,\textrm{
for $i=1,\ldots,d$}\}$ is the interior $P$ of $\bar{P}$. Now,
restricting $h$ to $h_K^{-1}(r)\cap(\mathbb{C}^*)^d$ gives a
$T^d$-bundle $h:h_K^{-1}(r)\cap(\mathbb{C}^*)^d\rightarrow P$ (which
is trivial), and $X$ is obtained by taking the quotient of this
$T^d$-bundle fiberwise by $T_K$. Hence, $X$ is naturally a
$T_N$-bundle over $P$, which can be written as
$$X=T^*P/N=P\times\sqrt{-1}T_N$$
(cf. Abreu \cite{Abreu00}).\footnote{We have, by abuse of notations,
used $N$ to denote the family of lattices $P\times\sqrt{-1}N$ over
$P$. Similarly, we denote by $M$ the family of lattices
$P\times\sqrt{-1} M$ below.} The reduced symplectic form
$\omega_X=\omega_{\bar{X}}|_X$ is the standard symplectic form
$$\omega_X=\sum_{j=1}^n dx_j\wedge du_j$$
where $x_1,\ldots,x_n\in\mathbb{R}$ and
$u_1,\ldots,u_n\in\mathbb{R}/2\pi\mathbb{Z}$ are respectively the
coordinates on $P\subset M_\mathbb{R}(r)$ and $T_N$. In other words,
the $x_j$'s and $u_j$'s are \textit{symplectic coordinates} (i.e.
action-angle coordinates). And the moment map is given by the
projection to $P$
$$\mu:X\rightarrow P.$$

We define the \textit{SYZ mirror manifold} by T-duality as follows.
\begin{defn}
The SYZ mirror manifold $Y_{SYZ}$ is defined as the total space of
the $T_M$-bundle, where $T_M=M_\mathbb{R}/M=(T_N)^\vee$, which is
obtained by fiberwise dualizing the $T_N$-bundle $\mu:X\rightarrow
P$.
\end{defn}
In other words, we have
$$Y_{SYZ}=TP/M=P\times\sqrt{-1}T_M\subset M_\mathbb{R}(r)\times\sqrt{-1}T_M.$$
$Y_{SYZ}$ has a natural complex structure, which is induced from the
one on $M_\mathbb{R}(r)\times\sqrt{-1}T_M\cong(\mathbb{C}^*)^n$. We
let $z_j=\exp(-x_j-\sqrt{-1}y_j)$, $j=1,\ldots,n$, be the
\textit{complex coordinates} on
$M_\mathbb{R}(r)\times\sqrt{-1}T_M\cong(\mathbb{C}^*)^n$ and
restricted to $Y_{SYZ}$, where
$y_1,\ldots,y_n\in\mathbb{R}/2\pi\mathbb{Z}$ are the coordinates on
$T_M=(T_N)^\vee$ dual to $u_1,\ldots,u_n$. We also let
$\Omega_{Y_{SYZ}}$ be the following nowhere vanishing holomorphic
$n$-form on $Y_{SYZ}$:
$$\Omega_{Y_{SYZ}}=\bigwedge_{j=1}^n(-dx_j-\sqrt{-1}dy_j)
=\frac{dz_1}{z_1}\wedge\ldots\wedge\frac{dz_n}{dz_n},$$ and denote
by
$$\nu:Y_{SYZ}\rightarrow P$$
the torus fibration dual to $\mu:X\rightarrow P$.
\begin{prop}\label{prop3.1}
The SYZ mirror manifold $Y_{SYZ}$ is contained in Hori-Vafa's mirror
manifold $Y_{HV}$ as an open complex submanifold. More precisely,
$Y_{SYZ}$ is the bounded domain $\{(Z_1,\ldots,Z_d)\in
Y_{SYZ}:|Z_i|<1,\ i=1,\ldots,d\}$ inside $Y_{HV}$.
\end{prop}
\begin{proof}
Dualizing the sequence (\ref{seq3.2}), we get
\begin{equation*}
1\longrightarrow T_M\overset{\check{\partial}}{\longrightarrow}
(T^d)^\vee\overset{\check{\iota}}{\longrightarrow}(T_K)^\vee\longrightarrow1,
\end{equation*}
while we also have the sequence (\ref{seq3.2'})
\begin{equation*}
0\longrightarrow
M_\mathbb{R}\overset{\check{\partial}}{\longrightarrow}(\mathbb{R}^d)^\vee
\overset{\check{\iota}}{\longrightarrow}K^\vee_\mathbb{R}\longrightarrow0.
\end{equation*}
Let $T_i=X_i+\sqrt{-1}Y_i\in\mathbb{C}/2\pi\sqrt{-1}\mathbb{Z}$,
$i=1,\ldots,d$, be the complex coordinates on
$(\mathbb{R}^d)^\vee\times\sqrt{-1}(T^d)^\vee\cong(\mathbb{C}^*)^d$.
If we let $Z_i=e^{-T_i}\in\mathbb{C}^*$, $i=1,\ldots,d$, then, by
the definition of $Y_{HV}$ and by (\ref{iota*}), we can identify
$Y_{HV}$ with the following complex submanifold in
$(\mathbb{C}^*)^d$:
\begin{equation*}
\check{\iota}^{-1}(r\in
K^\vee_\mathbb{R})\cap\check{\iota}^{-1}(1\in
(T_K)^\vee)\subset(\mathbb{R}^d)^\vee\times\sqrt{-1}(T^d)^\vee=(\mathbb{C}^*)^d.
\end{equation*}
Hence,
$$Y_{HV}=M_\mathbb{R}(r)\times\sqrt{-1}T_M\cong(\mathbb{C}^*)^n$$
as complex submanifolds in $(\mathbb{C}^*)^d$. Since
$Y_{SYZ}=TP/M=P\times\sqrt{-1}T_M$, $Y_{SYZ}$ is a complex
submanifold in $Y_{HV}$. In fact, as
$P=\check{\iota}^{-1}(r)\cap\{(X_1,\ldots,X_d)\in(\mathbb{R}^d)^*:X_i>0,i=1,\ldots,d\}$,
we have
$$Y_{SYZ}=\{(Z_1,\ldots,Z_d)\in Y_{HV}:|Z_i|<1,\textrm{ for $i=1,\ldots,d$}\}.$$
So $Y_{SYZ}$ is a bounded domain in $Y_{HV}$.
\end{proof}
We remark that, in terms of the complex coordinates
$z_j=\exp(-x_j-\sqrt{-1}y_j)$, $j=1,\ldots,n$, on
$Y_{HV}=M_\mathbb{R}(r)\times\sqrt{-1}T_M\cong(\mathbb{C}^*)^n$, the
embedding $\check{\partial}:Y_{HV}\hookrightarrow(\mathbb{C}^*)^d$
is given by
$$\check{\partial}(z_1,\ldots,z_n)=(e^{\lambda_1}z^{v_1},\ldots,e^{\lambda_d}z^{v_d}),$$
where $z^v$ denotes the monomial $z_1^{v^1}\ldots z_n^{v^n}$ if
$v=(v^1,\ldots,v^n)\in N=\mathbb{Z}^n$. So the coordinates $Z_i$'s
and $z_i$'s are related by
$$Z_i=e^{\lambda_i}z^{v_i},$$
for $i=1,\ldots,d$, and the SYZ mirror manifold $Y_{SYZ}$ is given
by the bounded domain
$$Y_{SYZ}=\{(z_1,\ldots,z_n)\in Y_{HV}=(\mathbb{C}^*)^n:|e^{\lambda_i}z^{v_i}|<1,\ i=1,\ldots,d\}.$$
Now, the superpotential $W:Y_{SYZ}\rightarrow\mathbb{C}$ (or
$W:Y_{HV}\rightarrow\mathbb{C}$) is of the form
$$W=e^{\lambda_1}z^{v_1}+\ldots+e^{\lambda_d}z^{v_d}.$$

From the above proposition, the SYZ mirror manifold $Y_{SYZ}$ is
strictly \textit{smaller} than Hori-Vafa's mirror manifold $Y_{HV}$.
This issue was discussed in Hori-Vafa \cite{HV00}, Section 3 and
Auroux \cite{Auroux07}, Section 4.2, and may be resolved by a
process called \textit{renormalization}. We refer the interested
reader to those references for the details. In this paper, we shall
always be (except in this subsection) looking at the SYZ mirror
manifold, and the letter $Y$ will also be used exclusively to denote
the SYZ mirror manifold.

\subsection{SYZ mirror transformations as fiberwise Fourier transforms}\label{subsec3.2}

In this subsection, we first give a brief review of semi-flat SYZ
mirror transformations (for details, see Hitchin \cite{Hitchin97},
Leung-Yau-Zaslow \cite{LYZ00} and Leung \cite{Leung00}). Then we
introduce the SYZ mirror transformations for toric Fano manifolds,
and prove part 1. of Theorem~\ref{main_thm}.\\

To begin with, recall that the dual torus $T_M=(T_N)^\vee$ of $T_N$
can be interpreted as the moduli space of flat $U(1)$-connections on
the trivial line bundle $T_N\times\mathbb{C}\rightarrow T_N$. In
more explicit terms, a point $y=(y_1,\ldots,y_n)\in
M_\mathbb{R}\cong\mathbb{R}^n$ corresponds to the flat
$U(1)$-connection $\nabla_y=d+\frac{\sqrt{-1}}{2}\sum_{j=1}^n y_j
du_j$ on the trivial line bundle $T_N\times\mathbb{C}\rightarrow
T_N$. The \textit{holonomy} of this connection is given, in our
convention, by the map
$$\textrm{hol}_{\nabla_y}:N\rightarrow U(1),\ v\mapsto
e^{-\sqrt{-1}\langle y,v\rangle}.$$ $\nabla_y$ is gauge equivalent
to the trivial connection $d$ if and only if
$(y_1,\ldots,y_n)\in(2\pi\mathbb{Z})^n=M$. So, in the following, we
will regard $(y_1,\ldots,y_n)\in
T_M\cong\mathbb{R}^n/(2\pi\mathbb{Z})^n$. Moreover, this
construction gives all flat $U(1)$-connections on
$T_N\times\mathbb{C}\rightarrow T_N$ up to unitary gauge
transformations. The universal $U(1)$-bundle $\mathcal{P}$, i.e. the
Poincar\'{e} line bundle, is the trivial line bundle over the
product $T_N\times T_M$ equipped with the $U(1)$-connection
$d+\frac{\sqrt{-1}}{2}\sum_{j=1}^n (y_jdu_j-u_jdy_j)$, where
$u_1,\ldots,u_n\in\mathbb{R}/2\pi\mathbb{Z}$ are the coordinates on
$T_N\cong\mathbb{R}^d/(2\pi\mathbb{Z})^d$. The curvature of this
connection is given by the two-form
$$F=\sqrt{-1}\sum_{j=1}^n dy_j\wedge du_j\in\Omega^2(T_N\times T_M).$$

From this perspective, the SYZ mirror manifold $Y$ is the moduli
space of pairs $(L_x,\nabla_y)$, where $L_x$ ($x\in P$) is a
Lagrangian torus fiber of $\mu:X\rightarrow P$ and $\nabla_y$ is a
flat $U(1)$-connection on the trivial line bundle
$L_x\times\mathbb{C}\rightarrow L_x$. The construction of the mirror
manifold in this way is originally advocated in the SYZ Conjecture
\cite{SYZ96} (cf. Hitchin \cite{Hitchin97} and Sections 2 and 4 in
Auroux \cite{Auroux07}).

Now recall that we have the dual torus bundles $\mu:X\rightarrow
P\textrm{ and }\nu:Y\rightarrow P$. Consider their fiber product
$X\times_P Y=P\times\sqrt{-1}(T_N\times T_M)$.
\begin{equation*}
\begin{CD}
  X\times_P Y   @>\pi_Y>>   Y \\
  @VV\pi_X V    @VV \nu V \\
  X             @>\mu>>   P
\end{CD}\\
\end{equation*}
By abuse of notations, we still use $F$ to denote the fiberwise
universal curvature two-form $\sqrt{-1}\sum_{j=1}^n dy_j\wedge
du_j\in\Omega^2(X\times_P Y)$.
\begin{defn}\label{def3.1}
The semi-flat SYZ mirror transformation
$\mathcal{F}^{\textrm{sf}}:\Omega^*(X)\rightarrow\Omega^*(Y)$ is
defined by
\begin{eqnarray*}
\mathcal{F}^{\textrm{sf}}(\alpha) & = &
(-2\pi\sqrt{-1})^{-n}\pi_{Y,*}(\pi_X^*(\alpha)\wedge e^{\sqrt{-1}F})\\
& = & (-2\pi\sqrt{-1})^{-n}\int_{T_N}\pi_X^*(\alpha)\wedge
e^{\sqrt{-1}F},
\end{eqnarray*}
where $\pi_X:X\times_P Y\rightarrow X$ and $\pi_Y:X\times_P
Y\rightarrow Y$ are the two natural projections.
\end{defn}
The key point is that, the semi-flat SYZ mirror transformation
$\mathcal{F}^{\textrm{sf}}$ transforms the (exponential of
$\sqrt{-1}$ times the) symplectic structure
$\omega_X=\sum_{j=1}^ndx_j\wedge du_j$ on $X$ to the holomorphic
$n$-form
$\Omega_Y=\frac{dz_1}{z_1}\wedge\ldots\wedge\frac{dz_n}{z_n}$ on
$Y$, where $z_j=\exp(-x_j-\sqrt{-1}y_j)$, $j=1,\ldots,n$. This is
probably well-known and implicitly contained in the literature, but
we include a proof here because we cannot find a suitable
reference.\\

\begin{prop}~\label{prop3.2}
We have
$$\mathcal{F}^{\textrm{sf}}(e^{\sqrt{-1}\omega_X})=\Omega_Y.$$
Moreover, if we define the inverse SYZ transformation
$(\mathcal{F}^{\textrm{sf}})^{-1}:\Omega^*(Y)\rightarrow\Omega^*(X)$
by
\begin{eqnarray*}
(\mathcal{F}^{sf})^{-1}(\alpha) & = &
(-2\pi\sqrt{-1})^{-n}\pi_{X,*}(\pi_Y^*(\alpha)\wedge e^{-\sqrt{-1}F})\\
& = & (-2\pi\sqrt{-1})^{-n}\int_{T_M}\pi_Y^*(\alpha)\wedge
e^{-\sqrt{-1}F},
\end{eqnarray*}
then we also have
$$(\mathcal{F}^{\textrm{sf}})^{-1}(\Omega_Y)=e^{\sqrt{-1}\omega_X}.$$
\end{prop}
\begin{proof}
The proof is by straightforward computations.
\begin{eqnarray*}
\mathcal{F}^{\textrm{sf}}(e^{\sqrt{-1}\omega_X}) & = &
(-2\pi\sqrt{-1})^{-n}\int_{T_N}\pi_{X}^*(e^{\sqrt{-1}\omega_X})\wedge e^{\sqrt{-1}F}\\
& = & (-2\pi\sqrt{-1})^{-n}\int_{T_N}e^{\sqrt{-1}\sum_{j=1}^n (dx_j+\sqrt{-1}dy_j)\wedge du_j}\\
& = & (-2\pi\sqrt{-1})^{-n}\int_{T_N} \bigwedge_{j=1}^n\big(1+\sqrt{-1}(dx_j+\sqrt{-1}dy_j)\wedge du_j\big)\\
& = & (2\pi)^{-n}\int_{T_N}
\Bigg(\bigwedge_{j=1}^n(-dx_j-\sqrt{-1}dy_j)\Bigg)\wedge
du_1\wedge\ldots\wedge du_n\\
& = & \Omega_Y,
\end{eqnarray*}
where we have $\int_{T_N}du_1\wedge\ldots\wedge du_n=(2\pi)^n$ for
the last equality. On the other hand,
\begin{eqnarray*}
(\mathcal{F}^{\textrm{sf}})^{-1}(\Omega_Y) & = &
(-2\pi\sqrt{-1})^{-n}\int_{T_M}\pi_Y^*(\Omega_Y)\wedge e^{-\sqrt{-1}F}\\
& = & (-2\pi\sqrt{-1})^{-n}\int_{T_M}
\Bigg(\bigwedge_{j=1}^n(-dx_j-\sqrt{-1}dy_j)\Bigg)\wedge
e^{\sum_{j=1}^n dy_j\wedge du_j}\\
& = & (2\pi\sqrt{-1})^{-n}\int_{T_M}\bigwedge_{j=1}^n\big((dx_j+\sqrt{-1}dy_j)\wedge e^{dy_j\wedge du_j}\big)\\
& = & (2\pi\sqrt{-1})^{-n}\int_{T_M}\bigwedge_{j=1}^n\big(dx_j+\sqrt{-1}dy_j-dx_j\wedge du_j\wedge dy_j\big)\\
& = & (2\pi)^{-n}\int_{T_M}\bigwedge_{j=1}^n \Big(1+\sqrt{-1}dx_j\wedge du_j\Big)\wedge dy_j\\
& = &
(2\pi)^{-n}\int_{T_M}\bigwedge_{j=1}^n\big(e^{\sqrt{-1}dx_j\wedge
du_j}\wedge dy_j\big)
\end{eqnarray*}
\begin{eqnarray*}
& = & (2\pi)^{-n}\int_{T_M}e^{\sqrt{-1}\sum_{j=1}^n dx_j\wedge du_j}\wedge dy_1\wedge\ldots\wedge dy_n\\
& = & e^{\sqrt{-1}\omega_X},
\end{eqnarray*}
where we again have $\int_{T_M}dy_1\wedge\ldots\wedge dy_n=(2\pi)^n$
in the last step.
\end{proof}
One can also apply the semi-flat SYZ mirror transformations to other
geometric structures and objects. For details, see Leung
\cite{Leung00}.\\

The semi-flat SYZ mirror transformation $\mathcal{F}^{\textrm{sf}}$
can transform the symplectic structure $\omega_X$ on $X$ to the
holomorphic $n$-form $\Omega_Y$ on $Y$. However, as we mentioned in
the introduction, we are not going to obtain the superpotential
$W:Y\rightarrow\mathbb{C}$ in this way because we have ignored the
toric boundary divisor $\bar{X}\setminus X=D_\infty=\bigcup_{i=1}^d
D_i$. Indeed, it is the toric boundary divisor $D_\infty$ which
gives rise to the quantum corrections in the A-model of $\bar{X}$.
More precisely, these quantum corrections are due to the existence
of holomorphic discs in $\bar{X}$ with boundary in Lagrangian torus
fibers which have intersections with the divisor $D_\infty$. To
restore this information, our way out is to look at the (trivial)
$\mathbb{Z}^n$-cover
$$\pi:LX=X\times N\rightarrow X.$$
Recall that we equip $LX$ with the symplectic structure
$\pi^*(\omega_X)$; we will confuse the notations and use $\omega_X$
to denote either the symplectic structure on $X$ or that on $LX$. We
will further abuse the notations by using $\mu$ to denote the
fibration
$$\mu:LX\rightarrow P,$$
which is the composition of the map $\pi:LX\rightarrow X$ with
$\mu:X\rightarrow P$.\\

We are now ready to define the SYZ mirror transformation
$\mathcal{F}$ for the toric Fano manifold $\bar{X}$. It will be
constructed as \textit{a combination of the semi-flat SYZ
transformation $\mathcal{F}^{\textrm{sf}}$ and taking fiberwise
Fourier series}.

Analog to the semi-flat case, consider the fiber product
$$LX\times_P Y=P\times N\times\sqrt{-1}(T_N\times T_M)$$
of the maps $\mu:LX\rightarrow P$ and $\nu:Y\rightarrow P$.
\begin{equation*}
\begin{CD}
  LX\times_P Y   @>\pi_Y>>   Y \\
  @VV\pi_{LX} V    @VV \nu V \\
  LX             @>\mu>>   P
\end{CD}\\
\end{equation*}
Note that we have a covering map $LX\times_P Y\rightarrow X\times_P
Y$. Pulling back $F\in\Omega^2(X\times_P Y)$ to $LX\times_P Y$ by
this covering map, we get the fiberwise universal curvature two-form
$$F=\sqrt{-1}\sum_{j=1}^n dy_j\wedge du_j\in\Omega^2(LX\times_P Y).$$
We further define the \emph{holonomy function}
$\textrm{hol}:LX\times_P Y\rightarrow U(1)$ as follows. For
$(p,v)\in LX$ and $z=(z_1,\ldots,z_n)\in Y$ such that
$\mu(p)=\nu(z)=:x\in P$, we let $x=(x_1,\ldots,x_n)$, and write
$z_j=\exp(-x_j-\sqrt{-1}y_j)$, so that $y=(y_1,\ldots,y_n)\in
(L_x)^\vee:=\nu^{-1}(x)\subset Y$. Then we set
$$\textrm{hol}(p,v,z):=\textrm{hol}_{\nabla_y}(v)=e^{-\sqrt{-1}\langle y,v\rangle},$$
where $\nabla_y$ is the flat $U(1)$-connection on the trivial line
bundle $L_x\times\mathbb{C}\rightarrow L_x$ over $L_x:=\mu^{-1}(x)$
corresponding to the point $y\in(L_x)^\vee$.
\begin{defn}\label{def3.2}
The SYZ mirror transformation
$\mathcal{F}:\Omega^*(LX)\rightarrow\Omega^*(Y)$ for the toric Fano
manifold $\bar{X}$ is defined by
\begin{eqnarray*}
\mathcal{F}(\alpha) & = &
(-2\pi\sqrt{-1})^{-n}\pi_{Y,*}(\pi_{LX}^*(\alpha)\wedge e^{\sqrt{-1}F}\textrm{hol})\\
& = & (-2\pi\sqrt{-1})^{-n}\int_{N\times
T_N}\pi_{LX}^*(\alpha)\wedge e^{\sqrt{-1}F}\textrm{hol},
\end{eqnarray*}
where $\pi_{LX}:LX\times_P Y\rightarrow LX$ and $\pi_Y:LX\times_P
Y\rightarrow Y$ are the two natural projections.
\end{defn}
Before stating the basic properties of $\mathcal{F}$, we introduce
the class of functions on $LX$ relevant to our applications.
\begin{defn}\label{def3.3} A $T_N$-invariant function
$f:LX\rightarrow\mathbb{C}$ is said to be admissible if for any
$(p,v)\in LX=X\times N$,
$$f(p,v)=f_v e^{-\langle x,v\rangle},$$
where $x=\mu(p)\in P$ and $f_v\in\mathbb{C}$ is a constant, and the
fiberwise Fourier series
\begin{equation*}
\widehat{f}:=\sum_{v\in N}f_v e^{-\langle
x,v\rangle}\textrm{hol}_{\nabla_y}(v)=\sum_{v\in N}f_vz^v,
\end{equation*}
where $z^v=\exp(\langle-x-\sqrt{-1}y,v\rangle)$, is convergent and
analytic, as a function on $Y$. We denote by $\mathcal{A}(LX)\subset
C^\infty(LX)$ set of all admissible functions on $LX$.
\end{defn}

Examples of admissible functions on $LX$ include those
$T_N$-invariant functions which are not identically zero on
$X\times\{v\}\subset LX$ for only finitely many $v\in N$. In
particular, the functions $\Psi_1,\ldots,\Psi_d$ are all in
$\mathcal{A}(LX)$. We will see (in the proof of
Theorem~\ref{thm3.2}) shortly that $\Phi_q$ is also admissible.

Now, for functions $f,g\in\mathcal{A}(LX)$, we define their
\textit{convolution product} $f\star g:LX\rightarrow\mathbb{C}$, as
before, by
$$(f\star g)(p,v)=\sum_{v_1,v_2\in N,\ v_1+v_2=v}f(p,v_1)g(p,v_2).$$
That the right-hand-side is convergent can be seen as follows. By
definition, $f,g\in\mathcal{A}(LX)$ implies that for any $p\in X$
and any $v_1,v_2\in N$,
$$f(p,v_1)=f_{v_1}e^{-\langle x,v_1\rangle},\ g(p,v_2)=g_{v_2}e^{-\langle x,v_2\rangle},$$
where $x=\mu(p)$ and $f_{v_1},g_{v_2}\in\mathbb{C}$ are constants;
also, the series $\widehat{f}=\sum_{v_1\in N}f_{v_1}z^{v_1}$ and
$\widehat{g}=\sum_{v_2\in N}g_{v_2}z^{v_2}$ are convergent and
analytic. Then their product, given by
\begin{eqnarray*}
\widehat{f}\cdot\widehat{g}=\Bigg(\sum_{v_1\in
N}f_{v_1}z^{v_1}\Bigg)\Bigg(\sum_{v_2\in
N}g_{v_2}z^{v_2}\Bigg)=\sum_{v\in N}\Bigg(\sum_{\substack{v_1,v_2\in
N,\\ v_1+v_2=v}}f_{v_1}g_{v_2}\Bigg)z^v,
\end{eqnarray*}
is also analytic. This shows that the convolution product $f\star g$
is well defined and gives another admissible function on $LX$.
Hence, the $\mathbb{C}$-vector space $\mathcal{A}(LX)$, together
with the convolution product $\star$, forms a
$\mathbb{C}$-algebra.\\

Let $\mathcal{O}(Y)$ be the $\mathbb{C}$-algebra of holomorphic
functions on $Y$. Recall that $Y=TP/M=P\times\sqrt{-1}T_M$. For
$\phi\in\mathcal{O}(Y)$, the restriction of $\phi$ to a fiber
$(L_x)^\vee=\nu^{-1}(x)\cong T_M$ gives a $C^\infty$ function
$\phi_x:T_M\rightarrow\mathbb{C}$ on the torus $T_M$. For $v\in N$,
the $v$-th Fourier coefficient of $\phi_x$ is given by
$$\widehat{\phi}_x(v)=\int_{T_M}\phi_x(y)e^{\sqrt{-1}\langle y,v\rangle}dy_1\wedge\ldots\wedge dy_n.$$
Then, we define a function $\widehat{\phi}:LX\rightarrow\mathbb{C}$
on $LX$ by
$$\widehat{\phi}(p,v)=\widehat{\phi}_x(v),$$
where $x=\mu(p)\in P$. $\widehat{\phi}$ is clearly admissible. We
call the process,
$\phi\in\mathcal{O}(Y)\mapsto\widehat{\phi}\in\mathcal{A}(LX)$,
taking \textit{fiberwise Fourier coefficients}. The following lemma
follows from the standard theory of Fourier analysis on tori (see,
for example, Edwards \cite{Edwards79}).
\begin{lem}\label{lem3.1}
Taking fiberwise Fourier series, i.e. the map
$$\mathcal{A}(LX)\rightarrow\mathcal{O}(Y),\quad f\mapsto\widehat{f}$$
is an isomorphism of $\mathbb{C}$-algebras, where we equip
$\mathcal{A}(LX)$ with the convolution product and $\mathcal{O}(Y)$
with the ordinary product of functions. The inverse is given by
taking fiberwise Fourier coefficients. In particular,
$\widehat{\widehat{f}}=f$ for any $f\in\mathcal{A}(LX)$.
\end{lem}
The basic properties of the SYZ mirror transformation $\mathcal{F}$
are summarized in the following theorem.
\begin{thm}\label{thm3.1}
Let
$\mathcal{A}(LX)e^{\sqrt{-1}\omega_X}:=\{fe^{\sqrt{-1}\omega_X}:f\in
\mathcal{A}(LX)\}\subset\Omega^*(LX)$ and
$\mathcal{O}(Y)\Omega_Y:=\{\phi\Omega_Y:\phi\in\mathcal{O}(Y)\}\subset\Omega^*(Y)$.
\begin{enumerate}

\item[(i)] For any admissible function $f\in\mathcal{A}(LX)$,
$$\mathcal{F}(fe^{\sqrt{-1}\omega_X})=\widehat{f}\Omega_Y\in\mathcal{O}(Y)\Omega_Y.$$

\item[(ii)] If we define the inverse SYZ mirror transformation
$\mathcal{F}^{-1}: \Omega^*(Y)\rightarrow\Omega^*(LX)$ by
\begin{eqnarray*}
\mathcal{F}^{-1}(\alpha) & = &
(-2\pi\sqrt{-1})^{-n}\pi_{LX,*}(\pi_Y^*(\alpha)\wedge e^{-\sqrt{-1}F}\textrm{hol}^{-1})\\
& = & (-2\pi\sqrt{-1})^{-n}\int_{T_M}\pi_Y^*(\alpha)\wedge
e^{-\sqrt{-1}F}\textrm{hol}^{-1},
\end{eqnarray*}
where $\textrm{hol}^{-1}:LX\times_P Y\rightarrow\mathbb{C}$ is the
function defined by
$\textrm{hol}^{-1}(p,v,z)=1/\textrm{hol}(p,v,z)=e^{\sqrt{-1}\langle
y,v\rangle}$, for any $(p,v,z)\in L_X\times_P Y$, then
$$\mathcal{F}^{-1}(\phi\Omega_Y)=\widehat{\phi}e^{\sqrt{-1}\omega_X}\in\mathcal{A}(LX)e^{\sqrt{-1}\omega_X},$$
for any $\phi\in\mathcal{O}(Y)$.

\item[(iii)] The restriction map
$\mathcal{F}:\mathcal{A}(LX)e^{\sqrt{-1}\omega_X}
\rightarrow\mathcal{O}(Y)\Omega_Y$ is a bijection with inverse
$\mathcal{F}^{-1}:\mathcal{O}(Y)\Omega_Y\rightarrow\mathcal{A}(LX)e^{\sqrt{-1}\omega_X}$,
i.e. we have
\begin{equation*}
\mathcal{F}^{-1}\circ\mathcal{F}=\textrm{Id}_{\mathcal{A}(LX)e^{\sqrt{-1}\omega_X}},\
\mathcal{F}\circ\mathcal{F}^{-1}=\textrm{Id}_{\mathcal{O}(Y)\Omega_Y}.
\end{equation*}
This shows that the SYZ mirror transformation $\mathcal{F}$ has the
inversion property.
\end{enumerate}
\end{thm}
\begin{proof}
Let $f\in\mathcal{A}(LX)$. Then, for any $v\in N$,
$f(p,v)=f_ve^{-\langle x,v\rangle}$ for some constant
$f_v\in\mathbb{C}$. By observing that both functions $\pi_{LX}^*(f)$
and $\textrm{hol}$ are $T_N$-invariant functions on $LX\times_P Y$,
we have
\begin{eqnarray*}
\mathcal{F}(fe^{\sqrt{-1}\omega_X}) & = &
(-2\pi\sqrt{-1})^{-n}\int_{N\times
T_N}\pi_{LX}^*(fe^{\sqrt{-1}\omega_X})\wedge
e^{\sqrt{-1}F}\textrm{hol}\\
& = & (-2\pi\sqrt{-1})^{-n}\sum_{v\in
N}\pi_{LX}^*(f)\cdot\textrm{hol}
\int_{T_N}\pi_{LX}^*(e^{\sqrt{-1}\omega_X})\wedge e^{\sqrt{-1}F}\\
& = & (-2\pi\sqrt{-1})^{-n}\Bigg(\sum_{v\in
N}\pi_{LX}^*(f)\cdot\textrm{hol}\Bigg)
\Bigg(\int_{T_N}\pi_X^*(e^{\sqrt{-1}\omega_X})\wedge
e^{\sqrt{-1}F}\Bigg).
\end{eqnarray*}
The last equality is due to the fact that the forms
$\pi_{LX}^*(e^{\sqrt{-1}\omega_X})=\pi_X^*(e^{\sqrt{-1}\omega_X})$
and $e^F$ are independent of $v\in N$. By Proposition~\ref{prop3.2},
the second factor is given by
$$\int_{T_N}\pi_X^*(e^{\sqrt{-1}\omega_X})\wedge e^{\sqrt{-1}F}
=(-2\pi\sqrt{-1})^n\mathcal{F}^{\textrm{sf}}(e^{\sqrt{-1}\omega_X})=(-2\pi\sqrt{-1})^n\Omega_Y,$$
while the first factor is the function on $Y$ given, for
$x=(x_1,\ldots,x_n)\in P$ and $y=(y_1,\ldots,y_n)\in T_M$, by
\begin{eqnarray*}
\Bigg(\sum_{v\in N}\pi_{LX}^*(f)\cdot\textrm{hol}\Bigg)(x,y) & = &
\sum_{v\in N} f_ve^{-\langle x,v\rangle}e^{-\sqrt{-1}\langle y,v\rangle}\\
& = & \sum_{v\in N} f_vz^v\\
& = & \widehat{f}(z),
\end{eqnarray*}
where
$z=(z_1,\ldots,z_n)=(\exp(-x_1-\sqrt{-1}y_1),\ldots,\exp(-x_n-\sqrt{-1}y_n))\in
Y$. Hence
$\mathcal{F}(fe^{\sqrt{-1}\omega_X})=\widehat{f}\Omega_Y\in\mathcal{O}(Y)\Omega_Y$.
This proves (i).

For (ii), expand $\phi\in\mathcal{O}(Y)$ into a fiberwise Fourier
series
$$\phi(z)=\sum_{w\in N}\widehat{\phi}_x(w)e^{-\sqrt{-1}\langle
y,w\rangle},$$ where $x,y,z$ are as before. Then
\begin{eqnarray*}
\mathcal{F}^{-1}(\phi\Omega_Y) & = &
(-2\pi\sqrt{-1})^{-n}\int_{T_M}\pi_Y^*(\phi\Omega_Y)\wedge e^{-\sqrt{-1}F}\textrm{hol}^{-1}\\
& = & (-2\pi\sqrt{-1})^{-n}\sum_{w\in
N}\Bigg(\widehat{\phi}_x(w)\int_{T_M}e^{\sqrt{-1}\langle
y,v-w\rangle}\pi_Y^*(\Omega_Y)\wedge e^{-\sqrt{-1}F}\Bigg).
\end{eqnarray*}
Here comes the key observation: If $v-w\neq0\in N$, then, using (the
proof of) the second part of Proposition~\ref{prop3.2}, we have
\begin{eqnarray*}
&   & \int_{T_M} e^{\sqrt{-1}\langle
y,v-w\rangle}\pi_Y^*(\Omega_Y)\wedge e^{-\sqrt{-1}F}\\
& = & \int_{T_M} e^{\sqrt{-1}\langle
y,v-w\rangle}\Bigg(\bigwedge_{j=1}^n(-dx_j-\sqrt{-1}dy_j)\Bigg)\wedge
e^{\sum_{j=1}^n dy_j\wedge du_j}\\
& = & (-\sqrt{-1})^n
e^{\sqrt{-1}\omega_X}\int_{T_M}e^{\sqrt{-1}\langle
y,v-w\rangle} dy_1\wedge\ldots\wedge dy_n\\
& = & 0.
\end{eqnarray*}
Hence,
\begin{eqnarray*}
\mathcal{F}^{-1}(\phi\Omega_Y) & = &
(-2\pi\sqrt{-1})^{-n}\widehat{\phi}_x(v)\int_{T_M}\pi_Y^*(\Omega_Y)\wedge
e^{-\sqrt{-1}F}\\
& = &
\widehat{\phi}\Big((\mathcal{F}^{\textrm{sf}})^{-1}(\Omega_Y)\Big)
=\widehat{\phi}e^{\omega_X}\in\mathcal{A}(LX)e^{\omega_X},
\end{eqnarray*}
again by Proposition~\ref{prop3.2}.

(iii) follows from (i), (ii) and Lemma~\ref{lem3.1}.
\end{proof}
We will, again by abuse of notations, also use
$\mathcal{F}:\mathcal{A}(LX)\rightarrow\mathcal{O}(Y)$ to denote the
process of taking fiberwise Fourier series:
$\mathcal{F}(f):=\widehat{f}$ for $f\in\mathcal{A}(LX)$. Similarly,
we use $\mathcal{F}^{-1}:\mathcal{O}(Y)\rightarrow\mathcal{A}(LX)$
to denote the process of taking fiberwise Fourier coefficients:
$\mathcal{F}^{-1}(\phi):=\widehat{\phi}$ for
$\phi\in\mathcal{O}(Y)$. To which meanings of the symbols
$\mathcal{F}$ and $\mathcal{F}^{-1}$ are we referring will be clear
from the context.

We can now prove the first part of Theorem~\ref{main_thm}, as a
corollary of Theorem~\ref{thm3.1}.
\begin{thm}[=part 1. of Theorem~\ref{main_thm}]\label{thm3.2}
The SYZ mirror transformation of the function $\Phi_q\in
C^\infty(LX)$, defined in terms of the counting of Maslov index two
holomorphic discs in $\bar{X}$ with boundary in Lagrangian torus
fibers, is the exponential of the superpotential $W$ on the mirror
manifold $Y$, i.e.
$$\mathcal{F}(\Phi_q)=e^W.$$
Conversely, we have
$$\mathcal{F}^{-1}(e^W)=\Phi_q.$$
Furthermore, we can incorporate the symplectic structure
$\omega_X=\omega_{\bar{X}}|_X$ on $X$ to give the holomorphic volume
form on the Landau-Ginzburg model $(Y,W)$ through the SYZ mirror
transformation $\mathcal{F}$, and vice versa, in the following
sense:
\begin{equation*}
\mathcal{F}(\Phi_q e^{\sqrt{-1}\omega_X})=e^W\Omega_Y,\
\mathcal{F}^{-1}(e^W\Omega_Y)=\Phi_q e^{\sqrt{-1}\omega_X}.
\end{equation*}
\end{thm}
\begin{proof}
By Theorem~\ref{thm3.1}, we only need to show that $\Phi_q\in
C^\infty(LX)$ is admissible and
$\mathcal{F}(\Phi_q)=\widehat{\Phi}_q=e^W\in\mathcal{O}(Y)$. Recall
that, for $(p,v)\in LX=X\times N$ and $x=\mu(p)\in P$,
$$\Phi_q(p,v)=\sum_{\beta\in\pi_2^+(\bar{X},L_x),\
\partial\beta=v}\frac{1}{w(\beta)}e^{-\frac{1}{2\pi}\int_\beta\omega_{\bar{X}}}.$$
For $\beta\in\pi_2^+(\bar{X},L_x)$ with $\partial\beta=v$, by the
symplectic area formula (\ref{area}) of Cho-Oh, we have
$\int_\beta\omega_{\bar{X}}=2\pi\langle x,v\rangle+\textrm{const}$.
So $\Phi_q(p,v)$ is of the form $\textrm{const}\cdot e^{-\langle
x,v\rangle}$. Now,
\begin{eqnarray*}
\sum_{v\in N}\Phi_q(p,v)\textrm{hol}_{\nabla_y}(v) & = & \sum_{v\in
N}\Bigg(\sum_{\beta\in\pi_2^+(\bar{X},L_x),\
\partial\beta=v}\frac{1}{w(\beta)}e^{-\frac{1}{2\pi}\int_\beta\omega_{\bar{X}}}\Bigg)e^{-\sqrt{-1}\langle y,v\rangle}\\
& = & \sum_{k_1,\ldots,k_d\in\mathbb{Z}_{\geq0}}\frac{1}{k_1!\ldots
k_d!}e^{-\sum_{i=1}^d k_i(\langle
x,v_i\rangle-\lambda_i)}e^{-\sum_{i=1}^d k_i\sqrt{-1}\langle
y,v_i\rangle}\\
& = & \prod_{i=1}^d\Bigg(\sum_{k_i=0}^\infty\frac{1}{k_i!}
\big(e^{\lambda_i-\langle
x+\sqrt{-1}y,v_i\rangle}\big)^{k_i}\Bigg)\\
& = & \prod_{i=1}^d \exp(e^{\lambda_i}z^{v_i})=e^W.
\end{eqnarray*}
This shows that $\Phi_q$ is admissible and $\widehat{\Phi}_q=e^W$.
\end{proof}
The form $\Phi_q e^{\sqrt{-1}\omega_X}\in\Omega^*(LX)$ can be viewed
as \textit{the symplectic structure modified by quantum corrections
from Maslov index two holomorphic discs in $\bar{X}$ with boundaries
on Lagrangian torus fibers}. That we call $e^W\Omega_Y$ the
holomorphic volume form of the Landau-Ginzburg model $(Y,W)$ can be
justified in several ways. For instance, in the theory of
singularities, one studies the complex oscillating integrals
$$I=\int_{\Gamma} e^{\frac{1}{\hbar}W}\Omega_Y,$$
where $\Gamma$ is some real $n$-dimensional cycle in $Y$ constructed
by the Morse theory of the function $\textrm{Re($W$)}$. These
integrals are reminiscent of the periods of holomorphic volume forms
on Calabi-Yau manifolds, and they satisfy certain Picard-Fuchs
equations (see, for example, Givental~\cite{Givental97b}). Hence,
one may think of $e^W\Omega_Y$ as playing the same role as the
holomorphic volume form on a Calabi-Yau manifold.

\subsection{Quantum cohomology vs. Jacobian ring}\label{subsec3.3}

The purpose of this subsection is to give a proof of the second part
of Theorem~\ref{main_thm}. Before that, let us recall the definition
of the Jacobian ring $Jac(W)$. Recall that the SYZ mirror manifold
$Y$ is given by the bounded domain
$$Y=\{(z_1,\ldots,z_n)\in(\mathbb{C}^*)^n:|e^{\lambda_i}z^{v_i}|<1,\ i=1,\ldots,d\},$$
in $(\mathbb{C}^*)^n$, and the superpotential
$W:Y\rightarrow\mathbb{C}$ is the Laurent polynomial
$$W=e^{\lambda_1}z^{v_1}+\ldots+e^{\lambda_d}z^{v_d},$$
where, as before, $z^v$ denotes the monomial $z_1^{v^1}\ldots
z_n^{v^n}$ if $v=(v^1,\ldots,v^n)\in N=\mathbb{Z}^n$. Let
$\mathbb{C}[Y]=\mathbb{C}[z_1^{\pm1},\ldots,z_n^{\pm1}]$ be the
$\mathbb{C}$-algebra of Laurent polynomials restricted to $Y$. Then
the Jacobian ring $Jac(W)$ of $W$ is defined as the quotient of
$\mathbb{C}[Y]$ by the ideal generated by the logarithmic
derivatives of $W$:
\begin{eqnarray*}
Jac(W) & = & \mathbb{C}[Y]\Big/\Big\langle z_j\frac{\partial
W}{\partial z_j}:j=1,\ldots,n\Big\rangle\\
& = &\mathbb{C}[z_1^{\pm1},\ldots,z_n^{\pm1}]\Big/\Big\langle
z_j\frac{\partial W}{\partial z_j}:j=1,\ldots,n\Big\rangle.
\end{eqnarray*}
The second part of Theorem~\ref{main_thm} is now an almost immediate
corollary of Proposition~\ref{prop2.2} and Theorem~\ref{thm3.3}.
\begin{thm}\label{thm3.3}
The SYZ mirror transformation $\mathcal{F}$ gives an isomorphism
$$\mathcal{F}:\mathbb{C}[\Psi_1^{\pm1},\ldots,\Psi_n^{\pm1}]/\mathcal{L}\rightarrow
Jac(W)$$ of $\mathbb{C}$-algebras. Hence, $\mathcal{F}$ induces a
natural isomorphism of $\mathbb{C}$-algebras between the small
quantum cohomology ring of $\bar{X}$ and the Jacobian ring of $W$:
\begin{eqnarray*}
\mathcal{F}:QH^*(\bar{X})\overset{\cong}{\longrightarrow}Jac(W),
\end{eqnarray*}
provided that $\bar{X}$ is a product of projective spaces.
\end{thm}
\begin{proof}
The functions $\Psi_1,\Psi_1^{-1},\ldots,\Psi_n,\Psi_n^{-1}$ are all
admissible, so $\mathbb{C}[\Psi_1^{\pm1},\ldots,\Psi_n^{\pm1}]$ is a
subalgebra of $\mathcal{A}(LX)$. It is easy to see that, for
$i=1,\ldots,d$, the SYZ mirror transformation
$\mathcal{F}(\Psi_i)=\widehat{\Psi}_i$ of $\Psi_i$ is nothing but
the monomial $e^{\lambda_i}z^{v_i}$. By our choice of the polytope
$\bar{P}\subset M_\mathbb{R}$, $v_1=e_1,\ldots,v_n=e_n$ is the
standard basis of $N=\mathbb{Z}^n$ and
$\lambda_1=\ldots=\lambda_n=0$. Hence,
$$\mathcal{F}(\Psi_i)=z_i,$$
for $i=1,\ldots,n$, and the induced map
$$\mathcal{F}:\mathbb{C}[\Psi_1^{\pm1},\ldots,\Psi_n^{\pm1}]\rightarrow\mathbb{C}[z_1^{\pm1},\ldots,z_n^{\pm1}]$$
is an isomorphism of $\mathbb{C}$-algebras. Now, notice that
$$z_j\frac{\partial W}{\partial z_j}
=\sum_{i=1}^d z_j\frac{\partial}{\partial
z_j}(e^{\lambda_i}z_1^{v_i^1}\ldots z_n^{v_i^n})=\sum_{i=1}^d v_i^j
e^{\lambda_i}z_1^{v_i^1}\ldots z_n^{v_i^n}=\sum_{i=1}^d v_i^j
e^{\lambda_i}z^{v_i},$$ for $j=1,\ldots,n$. The inverse SYZ
transformation of $z_j\frac{\partial W}{\partial z_j}$ is thus given
by
$$\mathcal{F}^{-1}(z_j\frac{\partial W}{\partial z_j})
=\widehat{\sum_{i=1}^d v_i^je^{\lambda_i}z^{v_i}}=\sum_{i=1}^d
v_i^j\Psi_i.$$ Thus,
$$\mathcal{F}^{-1}\Big(\Big\langle
z_j\frac{\partial W}{\partial
z_j}:j=1,\ldots,n\Big\rangle\Big)=\mathcal{L},$$ is the ideal in
$\mathbb{C}[\Psi_1^{\pm1},\ldots,\Psi_n^{\pm1}]$ generated by linear
equivalences. The result follows.
\end{proof}

\section{Examples}\label{sec4}

In this section, we give some examples to illustrate our results.\\

\noindent\textbf{Example 1. $\bar{X}=\mathbb{C}P^2$.} In this case,
$N=\mathbb{Z}^2$. The primitive generators of the 1-dimensional
cones of the fan $\Sigma$ defining $\mathbb{C}P^2$ are given by
$v_1=(1,0),v_2=(0,1),v_3=(-1,-1)\in N$, and the polytope
$\bar{P}\subset M_\mathbb{R}\cong\mathbb{R}^2$ we chose is defined
by the inequalities
$$x_1\geq0,\ x_2\geq0,\ x_1+x_2\leq t,$$
where $t>0$. See Figure 4.1 below.
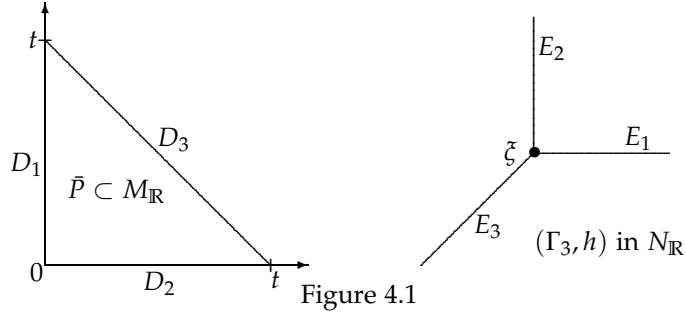
\begin{figure}[ht]
\setlength{\unitlength}{1mm}
\begin{picture}(100,35)
\put(10,2){\vector(0,1){35}} \put(10,2){\vector(1,0){35}}
\curve(10,32, 40,2) \curve(40,2.8, 40,1.2) \put(40,-1){$t$}
\curve(10.8,32, 9.2,32) \put(7.8,31){$t$} \put(8,0){0}
\put(13,11){$\bar{P}\subset M_\mathbb{R}$} \put(5.3,14.5){$D_1$}
\put(23,-1.3){$D_2$} \put(24.9,17.9){$D_3$} \curve(75,17, 93,17)
\curve(75,17, 75,35) \curve(75,17, 60,2)
\put(74.15,16.15){$\bullet$} \put(71,16){$\xi$} \put(87,18){$E_1$}
\put(75.5,30){$E_2$} \put(67,6.3){$E_3$} \put(75,4){$(\Gamma_3,h)$
in $N_\mathbb{R}$} \put(44,-3){Figure 4.1}
\end{picture}
\end{figure}

The mirror manifold $Y$ is given by
\begin{eqnarray*}
Y & = & \{(Z_1,Z_2,Z_3)\in\mathbb{C}^3:Z_1Z_2Z_3=q, |Z_i|<1,\ i=1,2,3\}\\
  & = & \{(z_1,z_2)\in(\mathbb{C}^*)^2:|z_1|<1, |z_2|<1,
  |\frac{q}{z_1z_2}|<1\},
\end{eqnarray*}
where $q=e^{-t}$ is the K\"{a}hler parameter, and, the
superpotential $W:Y\rightarrow\mathbb{C}$ can be written, in two
ways, as
$$W=Z_1+Z_2+Z_3=z_1+z_2+\frac{q}{z_1z_2}.$$
In terms of the coordinates $Z_1,Z_2,Z_3$, the Jacobian ring
$Jac(W)$ is given by
\begin{eqnarray*}
Jac(W) & = & \mathbb{C}[Z_1,Z_2,Z_3]\big/\big\langle
Z_1-Z_3,Z_2-Z_3,Z_1Z_2Z_3-q\big\rangle\\
& \cong & \mathbb{C}[Z]\big/\big\langle Z^3-q\big\rangle.
\end{eqnarray*}
There are three toric prime divisors $D_1,D_2,D_3$, which are
corresponding to the three admissible functions
$\Psi_1,\Psi_2,\Psi_3:LX\rightarrow\mathbb{R}$ defined by
\begin{eqnarray*}
\Psi_1(p,v) & = & \left\{ \begin{array}{ll}
                         e^{-x_1} & \textrm{if $v=(1,0)$}\\
                         0 & \textrm{otherwise,}
                         \end{array} \right.\\
\Psi_2(p,v) & = & \left\{ \begin{array}{ll}
                         e^{-x_2} & \textrm{if $v=(0,1)$}\\
                         0 & \textrm{otherwise,}
                         \end{array} \right.\\
\Psi_3(p,v) & = & \left\{ \begin{array}{ll}
                         e^{-(t-x_1-x_2)} & \textrm{if $v=(-1,-1)$}\\
                         0 & \textrm{otherwise,}
                         \end{array} \right.
\end{eqnarray*}
for $(p,v)\in LX$ and where $x=\mu(p)\in P$, respectively. The small
quantum cohomology ring of $\mathbb{C}P^2$ has the following
presentation:
\begin{eqnarray*}
QH^*(\mathbb{C}P^2) & = & \mathbb{C}[D_1,D_2,D_3]\big/\big\langle
D_1-D_3,D_2-D_3,D_1\ast D_2\ast D_3-q\big\rangle\\
& \cong & \mathbb{C}[H]\big/\big\langle H^3-q\big\rangle,
\end{eqnarray*}
where $H\in H^2(\mathbb{C}P^2,\mathbb{C})$ is the hyperplane class.
Quantum corrections appear only in one relation, namely,
$$D_1\ast D_2\ast D_3=q.$$
Fix a point $p\in X$. Then the quantum correction is due to the
unique holomorphic curve
$\varphi:(\mathbb{C}P^1;x_1,x_2,x_3,x_4)\rightarrow\mathbb{C}P^2$ of
degree 1 (i.e. a line) with 4 marked points such that
$\varphi(x_4)=p$ and $\varphi(x_i)\in D_i$ for $i=1,2,3$. The
parameterized 3-marked, genus 0 tropical curve corresponding to this
line is $(\Gamma_3;E_1,E_2,E_3;h)$, which is glued from three half
lines emanating from the point $\xi=\textrm{Log}(p)\in N_\mathbb{R}$
in the directions $v_1$, $v_2$ and $v_3$. See Figure 4.1 above.
These half lines are the parameterized Maslov index two tropical
discs $(\Gamma_1,h_i)$, where $h_i(V)=\xi$ an
$h_i(E)=\xi+\mathbb{R}_{\geq0}v_i$, for $i=1,2,3$ (see Figure 2.3).
They are corresponding to the Maslov index two holomorphic discs
$\varphi_1,\varphi_2,\varphi_3:(D^2,\partial
D^2)\rightarrow(\mathbb{C}P^2,L_{\mu(p)})$ which pass through $p$
and intersect the corresponding toric divisors $D_1,D_2,D_3$
respectively.\\

\noindent\textbf{Example 2.
$\bar{X}=\mathbb{C}P^1\times\mathbb{C}P^1$.} The primitive
generators of the 1-dimensional cones of the fan $\Sigma$ defining
$\mathbb{C}P^1\times\mathbb{C}P^1$ are given by
$v_{1,1}=(1,0),v_{2,1}=(-1,0),v_{1,2}=(0,1),v_{2,2}=(0,-1)\in
N=\mathbb{Z}^2$. We choose the polytope $\bar{P}\subset
M_\mathbb{R}=\mathbb{R}^2$ to be defined by the inequalities
$$0\leq x_1\leq t_1,\ 0\leq x_2\leq t_2$$
where $t_1,t_2>0$. See Figure 4.2 below.
\begin{figure}[ht]
\setlength{\unitlength}{1mm}
\begin{picture}(100,36)
\put(2,6){\vector(0,1){31}} \put(2,6){\vector(1,0){42}} \curve(2,32,
39,32) \curve(39,6, 39,32) \curve(39,6.8, 39,5.2)
\put(38,2.4){$t_1$} \curve(3.8,32, 6.2,32) \put(-1,31){$t_2$}
\put(0,4){0} \put(13.8,17.3){$\bar{P}\subset M_\mathbb{R}$}
\put(-4.5,18){$D_{1,1}$} \put(39.5,18){$D_{2,1}$}
\put(17.5,3.1){$D_{1,2}$} \put(17.5,33.2){$D_{2,2}$} \curve(57,23,
103,23) \curve(85,6, 85,35) \put(84.1,22){$\bullet$}
\put(82.8,19.7){$\xi$} \put(80,2.3){$(\Gamma_2,h_2)$}
\put(58.5,20){$(\Gamma_2,h_1)$} \put(92,30){$N_\mathbb{R}$}
\put(42,-2){Figure 4.2}
\end{picture}
\end{figure}

The mirror Landau-Ginzburg model $(Y,W)$ consists of
\begin{eqnarray*}
Y\!\!\!&=&\!\!\!\{(Z_{1,1},Z_{2,1},Z_{1,2},Z_{2,2})\in
\mathbb{C}^4:Z_{1,1}Z_{2,1}=q_1,Z_{1,2}Z_{2,2}=q_2,
|Z_{i,j}|<1,\textrm{ all $i,j$}\}\\
\!\!\!&=&\!\!\!\{(z_1,z_2)\in(\mathbb{C}^*)^2:|z_1|<1,|z_2|<1,
|\frac{q_1}{z_1}|<1,|\frac{q_2}{z_2}|<1\},
\end{eqnarray*}
where $q_1=e^{-t_1}$ and $q_2=e^{-t_2}$ are the K\"{a}hler
parameters, and
$$W=Z_{1,1}+Z_{2,1}+Z_{1,2}+Z_{2,2}=z_1+\frac{q_1}{z_1}+z_2+\frac{q_2}{z_2}.$$
The Jacobian ring $Jac(W)$ is given by
\begin{eqnarray*}
Jac(W) & = &
\frac{\mathbb{C}[Z_{1,1},Z_{2,1},Z_{1,2},Z_{2,2}]}{\big\langle
Z_{1,1}-Z_{2,1},Z_{1,2}-Z_{2,2},Z_{1,1}Z_{2,1}-q_1,Z_{1,2}Z_{2,2}-q_2\big\rangle}\\
& \cong & \mathbb{C}[Z_1,Z_2]\big/\big\langle
Z_1^2-q_1,Z_2^2-q_2\big\rangle.
\end{eqnarray*}
The four toric prime divisors $D_{1,1},D_{2,1},D_{1,2},D_{2,2}$
correspond respectively to the four admissible functions
$\Psi_{1,1},\Psi_{2,1},\Psi_{1,2},\Psi_{2,2}:LX\rightarrow\mathbb{C}$
defined by
\begin{eqnarray*}
\Psi_{1,1}(p,v) & = & \left\{ \begin{array}{ll}
                         e^{-x_1} & \textrm{if $v=(1,0)$}\\
                         0 & \textrm{otherwise,}
                         \end{array} \right.\\
\Psi_{2,1}(p,v) & = & \left\{ \begin{array}{ll}
                         e^{-(t_1-x_1)} & \textrm{if $v=(0,-1)$}\\
                         0 & \textrm{otherwise,}
                         \end{array} \right.\\
\Psi_{1,2}(p,v) & = & \left\{ \begin{array}{ll}
                         e^{-x_2} & \textrm{if $v=(0,1)$}\\
                         0 & \textrm{otherwise,}
                         \end{array} \right.\\
\Psi_{2,2}(p,v) & = & \left\{ \begin{array}{ll}
                         e^{-(t_2-x_2)} & \textrm{if $v=(0,-1)$}\\
                         0 & \textrm{otherwise,}
                         \end{array} \right.
\end{eqnarray*}
for $(p,v)\in LX$ and where $x=\mu(p)\in P$. The small quantum
cohomology ring of $\mathbb{C}P^1\times\mathbb{C}P^1$ is given by
\begin{eqnarray*}
QH^*(\mathbb{C}P^1\times\mathbb{C}P^1)\!\!\!&=&\!\!\!
\frac{\mathbb{C}[D_{1,1},D_{2,1},D_{1,2},D_{2,2}]}{\big\langle
D_{1,1}-D_{2,1},D_{1,2}-D_{2,2},D_{1,1}\ast D_{2,1}-q_1,D_{1,2}\ast
D_{2,2}-q_2\big\rangle}\\
\!\!\!&\cong&\!\!\!\mathbb{C}[H_1,H_2]\big/\big\langle
H_1^2-q_1,H_2^2-q_2\big\rangle
\end{eqnarray*}
where $H_1,H_2\in H^2(\mathbb{C}P^1\times\mathbb{C}P^1)$ are the
pullbacks of the hyperplane classes in the first and second factors
respectively. Quantum corrections appear in two relations:
$$D_{1,1}\ast D_{2,1}=q_1\textrm{ and }D_{1,2}\ast D_{2,2}=q_2.$$
Let us focus on the first one, as the other one is similar. For any
$p\in X$, there are two Maslov index two holomorphic discs
$\varphi_{1,1},\varphi_{2,1}:(D^2,\partial
D^2)\rightarrow(\mathbb{C}P^1\times\mathbb{C}P^1,L_{\mu(p)})$
intersecting the corresponding toric divisors. An interesting
feature of this example is that, since the sum of the boundaries of
the two holomorphic discs is zero as a \textit{chain}, instead of as
a class, in $L_{\mu(p)}$, they glue together \textit{directly} to
give the unique holomorphic curve
$\varphi_1:(\mathbb{C}P^1:x_1,x_2,x_3)\rightarrow\mathbb{C}P^1\times\mathbb{C}P^1$
of degree 1 with $\varphi_1(x_1)\in D_{1,1}$, $\varphi_1(x_2)\in
D_{2,1}$ and $\varphi_1(x_3)=p$. So the relation $D_{1,1}\ast
D_{2,1}=q_1$ is directly corresponding to
$\Psi_{1,1}\star\Psi_{2,1}=q_1\mathbb{1}$, without going through the
corresponding relation in $QH^*_{trop}(\bar{X})$. In other words, we
do not need to go to the tropical world to see the geometry of the
isomorphism $QH^*(\mathbb{C}P^1\times\mathbb{C}P^1)\cong
\mathbb{C}[\Psi_{1,1}^{\pm1},\Psi_{1,2}^{\pm1}]/\mathcal{L}$
(although in Figure 4.2 above, we have still drawn the tropical
lines $h_1$ and $h_2$ passing through $\xi=\textrm{Log}(p)\in
N_\mathbb{R}$).\\

\noindent\textbf{Example 3. $\bar{X}$ is the toric blowup of
$\mathbb{C}P^2$ at one point.} Let $\bar{P}\subset\mathbb{R}^2$ be
the polytope defined by the inequalities
$$x_1\geq0,\ 0\leq x_2\leq t_2,\ x_1+x_2\leq t_1+t_2,$$
where $t_1,t_2>0$.
\begin{figure}[ht]
\setlength{\unitlength}{1mm}
\begin{picture}(100,32)
\put(18,4){\vector(0,1){30}} \put(18,4){\vector(1,0){60}}
\curve(18,28, 44,28) \curve(18,28.1, 44,28.1) \curve(18,27.9,
44,27.9) \put(17.1,27.1){$\bullet$} \put(43.1,27.1){$\bullet$}
\curve(44,28, 68,4) \curve(18.8,28, 17.2,28) \put(14.1,27){$t_2$}
\curve(68,3.2, 68,4.8) \put(63.5,0.6){$t_1+t_2$} \put(16,1){0}
\put(28,14){$\bar{P}\subset M_\mathbb{R}$} \put(13.5,15.5){$D_1$}
\put(38,1){$D_2$} \put(56,17){$D_3$} \put(28,29){$D_4$}
\put(48,30.5){\vector(-1,0){15}} \put(49,29.5){exceptional curve}
\put(80,15){Figure 4.3}
\end{picture}
\end{figure}

The toric Fano manifold $\bar{X}$ corresponding to this trapezoid
(see Figure 4.3 above) is the blowup of $\mathbb{C}P^2$ at a
$T_N$-fixed point. The primitive generators of the 1-dimensional
cones of the fan $\Sigma$ defining $\bar{X}$ are given by
$v_1=(1,0),v_2=(0,1),v_3=(-1,-1),v_4=(0,-1)\in N=\mathbb{Z}^2$. As
in the previous examples, we have the mirror manifold
\begin{eqnarray*}
Y & = &
\{(Z_1,Z_2,Z_3,Z_4)\in\mathbb{C}^4:Z_1Z_3=q_1Z_4,Z_2Z_4=q_2,|Z_i|<1\textrm{for
all $i$}\}\\
& = &
\{(z_1,z_2)\in(\mathbb{C}^*)^2:|z_1|<1,|z_2|<1,|\frac{q_1q_2}{z_1z_2}|<1,|\frac{q_2}{z_2}|<1\},
\end{eqnarray*}
and the superpotential
$$W=Z_1+Z_2+Z_3+Z_4=z_1+z_2+\frac{q_1q_2}{z_1z_2}+\frac{q_2}{z_2},$$
where $q_1=e^{-t_1},q_2=e^{-t_2}$. The Jacobian ring of $W$ is
\begin{equation*}
Jac(W)=\frac{\mathbb{C}[Z_1,Z_2,Z_3,Z_4]}{\big\langle
Z_1-Z_3-Z_4,Z_2-Z_4,Z_1Z_3-q_1Z_4,Z_2Z_4-q_2\big\rangle}
\end{equation*}
and the small quantum cohomology ring of $\bar{X}$ is given by
$$QH^*(\bar{X})=\frac{\mathbb{C}[D_1,D_2,D_3,D_4]}{\big\langle
D_1-D_3-D_4,D_2-D_4,D_1\ast D_3-q_1D_4,D_2\ast
D_4-q_2\big\rangle}.$$ Obviously, we have an isomorphism
$QH^*(\bar{X})\cong Jac(W)$ and, the isomorphism
$$QH^*(\bar{X})\cong\mathbb{C}[\Psi_1^{\pm1},\Psi_2^{\pm1}]/\mathcal{L}$$
in Proposition~\ref{prop2.2} still holds, as we have said in
Remark~\ref{rmk2.3}. However, the geometric picture that we have
derived in Subsection~\ref{subsec2.2} using tropical geometry breaks
down. This is because there is a \textit{rigid} holomorphic curve
contained in the toric boundary $D_\infty$ which \textit{does}
contribute to $QH^*(\bar{X})$. Namely, the quantum relation
$$D_1\ast D_3=q_1D_4$$
is due to the holomorphic curve
$\varphi:\mathbb{C}P^1\rightarrow\bar{X}$ such that
$\varphi(\mathbb{C}P^1)\subset D_4$. This curve is exceptional since
$D_4^2=-1$, and thus cannot be deformed to a curve outside the toric
boundary. See Figure 4.3 above. Hence, it is \textit{not}
corresponding to any tropical curve in $N_\mathbb{R}$. This means
that tropical geometry cannot "see" the curve $\varphi$, and it is
not clear how one could define the tropical analog of the small
quantum cohomology ring in this case.

\section{Discussions}

In this final section, we speculate the possible generalizations of
the results of this paper. The discussion will be rather informal.\\

The proofs of the results in this paper rely heavily on the
classification of holomorphic discs in a toric Fano manifold
$\bar{X}$ with boundary in Lagrangian torus fibers, and on the
explicit nature of toric varieties. Nevertheless, it is still
possible to generalize these results, in particular, the
construction of SYZ mirror transformations, to non-toric situations.
For example, one may consider a \textit{complex flag manifold}
$\bar{X}$, where the Gelfand-Cetlin integrable system provide a
natural Lagrangian torus fibration structure on $\bar{X}$ (see, for
example, Guillemin-Sternberg \cite{GS83}). The base of this
fibration is again an affine manifold with boundary but without
singularities. In fact, there is a \textit{toric degeneration} of
the complex flag manifold $\bar{X}$ to a toric variety, and the base
is nothing but the polytope associated to that toric variety.
Furthermore, the classification of holomorphic discs in a complex
flag manifold $\bar{X}$ with boundary in Lagrangian torus fibers was
recently done by Nishinou-Nohara-Ueda \cite{NNU08}, and, at least
for the full flag manifolds, there is an isomorphism between the
small quantum cohomology ring and the Jacobian ring of the mirror
superpotential (cf. Corollary 12.4 in \cite{NNU08}). Hence, one can
try to construct the SYZ mirror transformations for a complex flag
manifold $\bar{X}$ and prove results like Proposition \ref{prop1.1}
and Theorem \ref{main_thm} as in the toric Fano case.\\

Certainly, the more important (and more ambitious) task is to
generalize the constructions of SYZ mirror transformations to the
most general situations, where the bases of Lagrangian torus
fibrations are affine manifolds with both boundary and
singularities. To do this, the first step is to make the
construction of the SYZ mirror transformations become a local one.
One possible way is the following: Suppose that we have an
$n$-dimensional compact K\"{a}hler manifold $\bar{X}$, together with
an anticanonical divisor $D$. Assume that there is a Lagrangian
torus fibration $\mu:\bar{X}\rightarrow\bar{B}$, where $\bar{B}$ is
a real $n$-dimensional (possibly) singular affine manifold with
boundary $\partial\bar{B}$. We should also have
$\mu^{-1}(\partial\bar{B})=D$. Now let $U\subset
B:=\bar{B}\setminus\partial\bar{B}$ be a small open ball contained
in an affine chart of the nonsingular part of $B$, i.e.
$\mu^{-1}(b)$ is a nonsingular Lagrangian torus in $\bar{X}$ for any
$b\in U$, so that we can identify each fiber $\mu^{-1}(b)$ with
$T^n$ and identify $\mu^{-1}(U)$ with $T^*U/\mathbb{Z}^n\cong
U\times T^n$. Let $N\cong\mathbb{Z}^n$ be the fundamental group of
any fiber $\mu^{-1}(b)$, and consider the $\mathbb{Z}^n$-cover
$L\mu^{-1}(U)=\mu^{-1}(U)\times N$. Locally, the mirror manifold
should be given by the dual torus fibration
$\nu:U\times(T^n)^\vee\rightarrow U$. Denote by $\nu^{-1}(U)$ the
local mirror $U\times(T^n)^\vee$. Then we can define the
\textit{local SYZ mirror transformation}, as before, through the
fiber product $L\mu^{-1}(U)\times_U\nu^{-1}(U)$.
\begin{equation*}
\begin{CD}
  L\mu^{-1}(U)\times_U\nu^{-1}(U)   @> >>   \nu^{-1}(U) \\
  @VV V    @VV \nu V \\
  L\mu^{-1}(U)             @>\mu>>   U
\end{CD}\\
\end{equation*}
Now, fix a reference fiber $L_0=\mu^{-1}(b_0)$. Given $v\in N$,
define a function $\Upsilon_v:L\mu^{-1}U\rightarrow\mathbb{R}$ as
follows. For any point $p\in\mu^{-1}(U)$, let $L_b=\mu^{-1}(b)$ be
the fiber containing $p$, where $b=\mu(p)\in U$. Regard $v$ as an
element in $\pi_1(L_b)$. Consider the 2-chain $\gamma$ in
$\mu^{-1}(U)$ with boundary in $v\cup L_0$, and define
$$\Upsilon_v(p,v)=\exp(-\frac{1}{2\pi}\int_\gamma\omega_{\mu^{-1}(U)}),$$
where $\omega_{\mu^{-1}(U)}=\omega_{\bar{X}}|_{\mu^{-1}(U)}$ is the
restriction of the K\"{a}hler form to $\mu^{-1}(U)$. Also set
$\Upsilon_v(p,w)=0$ for any $w\in N\setminus\{v\}$ (cf. the
discussion after Lemma 2.7 in Auroux \cite{Auroux07}). This is
analog to the definitions of the functions $\Psi_1,\ldots,\Psi_d$ in
the toric Fano case, and it is easy to see that the local SYZ mirror
transformations of these kind of functions give local holomorphic
functions on the local mirror $\nu^{-1}(U)=U\times(T^n)^\vee$. We
expect that these constructions will be sufficient for the purpose
of understanding quantum corrections due to the boundary divisor
$D$. However, to take care of the quantum corrections which arise
from the proper singular Lagrangian fibers (i.e. singular fibers
contained in $X=\mu^{-1}(B)$), one must modify and generalize the
constructions of the local SYZ mirror transformations to the case
where $U\subset B$ contains singular points. For this, new ideas are
needed in order to incorporate the wall-crossing phenomena.\\

\noindent\textbf{Acknowledgements.} We thank the referees for very
useful comments and suggestions. The work of the second author was
partially supported by RGC grants from the Hong Kong Government.

\end{document}